%% file: nondeg-rep.tex
\documentclass[12pt]{article}
\usepackage{amsmath,amssymb,amsthm,latexsym,amscd,stmaryrd,mathrsfs,longtable,color,url}

\newtheoremstyle{standard}%
{9pt}%
{9pt}%
{\it}
{}%
{\bfseries}%
{}
{ }%
{#3}%

\newcommand{\db}[1]{(\!({#1})\!)}

\newcommand{\subL}{K}

\numberwithin{equation}{section}

\newcommand{\N}{{\mathbb N}}
\newcommand{\Z}{{\mathbb Z}}
\newcommand{\Q}{{\mathbb Q}}

\newcommand{\C}{{\mathbb C}}

\newcommand{\wh}{h}
\newcommand{\Har}{H}

\newcommand{\wx}{x}

\newcommand{\rankL}{d}
\newcommand{\mN}{N}
\newcommand{\mW}{W}

\newcommand{\module}{M}

\newcommand{\mn}{p}
\newcommand{\hei}{{\mathfrak h}}

\newcommand{\wideomega}{\omega}
\newcommand{\wideHar}{\Har}

\newcommand{\zero}{\mathbf 0}

\newcommand{\fh}{{\mathfrak h}}

\newcommand{\lu}{u}
\newcommand{\lw}{w}
\newcommand{\lv}{v}

\newcommand{\mK}{K}

\newcommand{\lE}{t}
\newcommand{\ExB}{E}

\newcommand{\sv}{P}
\newcommand{\vac}{{\mathbf 1}}
\newcommand{\lattice}{L}

\newcommand{\nS}{S}

\newcommand{\Harfour}{\Har}
\newcommand{\Harfourone}{\Har^{[1]}}
\newcommand{\Harsix}{\Har^{\langle 6\rangle}}

\DeclareMathOperator{\Ext}{Ext} 
\DeclareMathOperator{\id}{id}

\DeclareMathOperator{\Span}{Span} 
\DeclareMathOperator{\wt}{wt}

\DeclareMathOperator{\rank}{rank}

\newtheorem{lemma}{Lemma}[section]
\newtheorem{theorem}[lemma]{Theorem}

\newtheorem{corollary}[lemma]{Corollary}
\theoremstyle{definition}
\newtheorem{definition}[lemma]{Definition}

\newtheorem{remark}[lemma]{Remark}

\theoremstyle{standard}

\title{Representations of the fixed point subalgebra of the vertex algebra associated to
a non-degenerate even lattice by an automorphism of order $2$}
\author{Kenichiro Tanabe\footnote{Research was partially supported by the Grant-in-aid
(No. 18K03198) for Scientific Research, JSPS.}\\\\
Department of Mathematics\\
Hokkaido University\\
Kita 10, Nishi 8, Kita-Ku, Sapporo, Hokkaido, 060-0810\\
Japan\\\\
ktanabe@math.sci.hokudai.ac.jp}
\date{}
\begin{document}
\maketitle
\begin{abstract}
Let $V_{\lattice}$ be the vertex algebra  associated to a non-degenerate even lattice $\lattice$,
$\theta$ the automorphism of $V_{\lattice}$ induced from the $-1$-isometry of $\lattice$, and
$V_{\lattice}^{+}$ the fixed point subalgebra of $V_{\lattice}$ under the action of $\theta$.
We show that every  weak $V_{\lattice}^{+}$-module is completely reducible.
\end{abstract}
\section{Introduction} 
\textcolor{black}{The vertex (operator) algebras} $V$ such that every ($\N$-graded) weak $V$-module is completely reducible are a fundamental class of vertex (operator) algebras
as with groups, rings, and Lie algebras.
In the case of weak modules, a typical example of such a vertex algebra is the vertex algebra $V_{\lattice}$ associated to a non-degenerate even lattice $\lattice$ \cite{Dong1993,DLM1997,FLM}.
A vertex operator algebra $V$ is called {\it rational} (resp. {\it regular}) if every $\N$-graded weak (resp. weak) $V$-module is completely reducible.
Thus, the vertex operator algebra $V_{\lattice}$ associated to a positive definite lattice $\lattice$ is an example of regular vertex operator algebra,
where we note that the vertex algebra $V_{\lattice}$ is a vertex operator algebra
only in the case when $L$ is positive definite.
Examples of other regular, and hence rational vertex operator algebras 
include the moonshine vertex operator algebra \textcolor{black}{$V^{\natural}$} \cite{Dong1994-moonshine, FLM},
vertex operator algebras associated to the irreducible vacuum representations for affine Kac--Moody algebras with positive integer levels \cite{DLM1997, Frenkel-Zhu1992},
vertex operator algebras of \textcolor{black}{irreducible highest weight modules for the Virasoro algebra of highest weight zero and minimal central charges} \cite{DLM1997, Frenkel-Zhu1992, Wang1993},
the vertex operator algebras associated to minimal representations of $W$-algebras \cite{arakawa2015}.
In most cases above, the Zhu algebras, which are associative $\C$-algebras introduced in \cite{Z1996},
play an essential role in showing the rationality since \cite[Theorem 2.2.1]{Z1996} shows that
for a vertex operator algebra $V$ there is a one to one correspondence between the set of all isomorphism classes of irreducible $\N$-graded weak $V$-modules 
and that of irreducible modules for the Zhu algebra associated to $V$.
We remark that
 if a vertex operator algebra $V=\oplus_{n=0}^{\infty}V_n$ with $V_{0}=\C\vac$
is regular, then every weak $V$-module is $\N$-graded by \cite[Theorem 4.5 and Proposition 5.7]{ABD2004}.

It is expected that the property of a simple vertex operator algebra of being rational is preserved by taking the fixed point subalgebra
under the action of any finite automorphism group.
Namely, let $V$ be a vertex algebra, $G$ a finite automorphism group of $V$,
and $V^G$ the fixed point subalgebra of $V$ under the action of $G$: $V^{G}=\{u\in V\ |\ gu=u\mbox{ for all }g\in G\}$.
It is conjectured that if a simple vertex operator algebra $V$ is rational, then so is $V^{G}$ \cite{DVVV1989}.
This conjecture is an analogue of a result due to Levitzki (\cite{Levitzki1935}, see also \cite[Theorem 1.15]{Montgomery1980}) in ring theory.
Let  $\lattice$ be a non-degenerate even lattice and $\theta$ the automorphism of $V_{\lattice}$ induced from the $-1$-isometry of $\lattice$.
The conjecture has been confirmed in the vertex operator algebra $V_{\lattice}^{+}:=V_{\lattice}^{\langle\theta \rangle}$ \cite{Abe2005, DJL2012}, where $L$ is a positive definite even lattice.
It is worth mentioning that the moonshine vertex operator algebra 
$V^{\natural}$ is constructed as a direct sum of $V_{\Lambda}^{+}$ and an irreducible $V_{\Lambda}^{+}$-module
in \cite{FLM} where $\Lambda$ is the Leech lattice.
This example demonstrates the importance of the fixed point vertex subalgebras, in particular $V_{\lattice}^{+}$.
When $\lattice$ is not positive definite, the vertex algebra $V_{\lattice}$ is not a vertex operator algebra as mentioned above, however,
it is established  in \cite[Theorems 3.8 and 3.16]{Yamsk2009} that every $\N$-graded weak $V_{\lattice}^{+}$-module is completely reducible by using the Zhu algebra.
There are some other examples in which the conjecture is confirmed (cf. \cite{DLaTYY2004, TY2007, TY2013}).

\textcolor{black}{Contrary} to the case of $\N$-graded weak modules, the study of non-$\N$-graded weak modules 
 for vertex (operator) algebras has not progressed
and is widely known to be difficult because of the absence of useful tool like the Zhu algebras for non-$\N$-graded weak modules. 
So far, the vertex algebras $V_{\lattice}$ such that $\lattice$ are not positive definite are only one class of vertex algebras which are not vertex operator algebras and whose weak modules are well understood.
For some vertex operator algebras, the irreducible weak modules with Whittaker vectors, which are non-$\N$-graded, are classified in \cite{Tanabe2017,Tanabe2020}.  

The aim of this paper is the study of weak $V_{\lattice}^{+}$-modules for any non-degenerate even lattice $\lattice$ of finite rank.
More precisely, as a continuation of  \cite[Theorem 1.1]{Tanabe2019} which classifies the irreducible weak $V_{\lattice}^{+}$-modules,
we shall show the following result, which means that an analogue of the conjecture for 
the fixed point vertex operator subalgebras above holds for the vertex algebra $V_{\lattice}^{+}$:
\begin{theorem}
\label{theorem:main}
Every weak $V_{\lattice}^{+}$-module is completely reducible
for any non-degenerate even lattice $\lattice$ of finite rank.
\end{theorem}
When $\lattice$ is positive definite, this result is shown in \cite[Theorem 7.7]{Abe2005} in the case that $\rank\lattice=1$ and in  
\cite[Theorem 6.5]{DJL2012} in the general case.

Let us explain the basic idea briefly.
Let $\lattice$ be a non-degenerate even lattice.
Since \cite[Lemma 7.3]{Tanabe2019} shows that 
every weak $V_{\lattice}^{+}$-module has an irreducible weak $V_{\lattice}^{+}$-module,
it is enough to show that $\Ext^{1}_{V_{\lattice}^{+}}(\module,\mW)=0$
for any pair of irreducible weak $V_{\lattice}^{+}$-modules $\module$ and $\mW$.
Let $M(1)$ be the Heisenberg vertex operator algebra constructed from $\fh:=\C\otimes_{\Z}\lattice$.
Then, $M(1)$ is a subalgebra of $V_{\lattice}$ and 
the fixed point subspace $M(1)^{+}=\{a\in M(1)\ |\ \theta(a)=a\}$ is a subalgebra of $V_{\lattice}^{+}$.
The irreducible $M(1)^{+}$-modules are classified in \cite[Theorem 4.5]{DN1999-1} in the case that $\dim_{\C}\fh=1$ and in \cite[Theorem 6.2.2]{DN2001} in the general case.
We have classified the irreducible weak $V_{\lattice}^{+}$-modules in  \cite[Theorem 1.1]{Tanabe2019}
and we remark that each irreducible weak $V_{\lattice}^{+}$-module is a direct sum of irreducible $M(1)^{+}$-modules.
We follow the notation in these results.
We first show that for any irreducible $M(1)^{+}$-submodule $\mK$ of a weak $V_{\lattice}^{+}$-module
which is not isomorphic to $M(1,\lambda)$ for any $\lambda\in\lattice$,
the weak $V_{\lattice}^{+}$-submodule generated by $\mK$ is completely reducible (Lemma \ref{lemma:generate-completely}).
We note that each  irreducible weak $V_{\lattice}^{+}$-module includes such an irreducible $M(1)^{+}$-submodule $\mK$.
Thus, it is enough to show that for any pair of irreducible weak $V_{\lattice}^{+}$-modules $\module$ and $\mW$,
any exact sequence of weak $V_{\lattice}^{+}$-modules 
\begin{align}
\label{eq:intor-exact-sequence}
0\rightarrow \mW\rightarrow \mN\overset{\pi}{\rightarrow} \module\rightarrow 0,
\end{align}
and the irreducible $M(1)^{+}$-submodule $\mK$ of $\module$ in Lemma \ref{lemma:generate-completely},
there exists an $M(1)^{+}$-submodule of $\mN$ which is isomorphic to $\mK$ under the restriction of $\pi$ in \eqref{eq:intor-exact-sequence}.
When $(\module,\mW)$ is not isomorphic to  $(V_{\lattice}^{\pm},V_{\lattice}^{\mp})$, $(V_{\lambda+\lattice},V_{\lambda+\lattice})$ with $\lambda\in\lattice^{\perp}$ and $2\lambda\not\in\lattice$,
$(V_{\lambda+\lattice}^{\rho},V_{\lambda+\lattice}^{\sigma})$ with $\lambda\in\lattice^{\perp}\setminus\lattice$, $2\lambda\in\lattice$, and $\rho,\sigma\in\{+,-\}$,
we can carry out this procedure using results in \cite{Tanabe2019} and \cite{Yamsk2009} (Corollary \ref{corollary:many-ext=0}).
%
For the other $(\module,\mW)$, we take a certain preimage $U$ of the lowest space of $\mK$ under the map $\pi : \mN\rightarrow \module$ in \eqref{eq:intor-exact-sequence}.
If $(\module,\mW)\not\cong(V_{\lattice}^{+},V_{\lattice}^{-})$, then a similar argument as in the proof of \cite[Lemmas 3.10 and 5.5]{Tanabe2019} shows that $U$ is a subspace of 
\begin{align}
\Omega_{M(1)^{+}}(\mN)&=\Big\{\lu\in N\ \Big|\ 
\begin{array}{l}
a_{i}\lu=0\ \mbox{for all homogeneous }a\in V\\
\mbox{and }i>\wt a-1.
\end{array}\Big\}
\end{align}
and is an irreducible module for the Zhu algebra $A(M(1)^{+})$.
It follows from \cite[Corollary 5.9]{Tanabe2019} that the $M(1)^{+}$-submodule of $\mN$ generated by $U$
is isomorphic to $\mK$ (Lemmas \ref{lemma:2lambda(not)inlattice} and \ref{lemma:Ext1V(Vlattice-,V\lattice+}).
If $(\module,\mW)\cong(V_{\lattice}^{+},V_{\lattice}^{-})$, then $U$ is spanned by a preimage $\lu$ of the vacuum element $\vac\in\module$.
We show that $\omega_{0}\lu=0$ and hence the $V_{\lattice}^{+}$-submodule of $\mN$ generated by $\lu$
is isomorphic to $V_{\lattice}^{+}$ (Lemma \ref{lemma:Ext1V(Vlattice+,V\lattice-}). 

Throughout this paper, complicated computation has been done by a computer algebra system Risa/Asir\cite{Risa/Asir}.

The organization of the paper is as follows. 
In Section \ref{section:preliminary} we recall some basic properties of the
Heisenberg vertex operator algebra $M(1)$,
the vertex algebra $V_{\lattice}$ associated to a non-degenerate even lattice $\lattice$,
and the fixed point subalgebras $M(1)^{+}$ and $V_{\lattice}^{+}$.
In Section \ref{section:main-result} we give a proof
of Theorem \ref{theorem:main}.

\section{\label{section:preliminary}Preliminary}
We assume that the reader is familiar with the basic knowledge on
vertex algebras as presented in \cite{B1986,FLM,LL,Li1996}. 

Throughout this paper, $\mn$ is a non-zero integer,
$\N$ denotes the set of all non-negative integers,
$\Z$ denotes the set of all integers, 
$\lattice$ is a non-degenerate even lattice of finite rank $\rankL$ with a bilinear form $\langle\ ,\ \rangle$,
$\lattice^{\perp}$ is the dual of $\lattice$: $\lattice^{\perp}=\{\alpha\in\Q\otimes_{\Z}\lattice\ |\ \langle\alpha,\beta\rangle\in\Z\mbox{ for all }\beta\in\lattice\}$,
and $(V,Y,{\mathbf 1})$ is a vertex algebra.
Recall that $V$ is the underlying vector space, 
$Y(\mbox{ },\wx)$ is the linear map from $V\otimes_{\C}V$ to $V\db{x}$, and
${\mathbf 1}$ is the vacuum vector.
Throughout this paper, we always assume that $V$ has an element $\omega$ such that $\omega_{0}a=a_{-2}\vac$ for all $a\in V$.
For a vertex operator algebra $V$, this condition automatically holds since $V$ has the conformal vector (Virasoro element).
We write down the definition of a weak $V$-module:
\begin{definition}
\label{definition:weak-module}
A {\it weak $V$-module} $\module$ is a vector space over $\C$ equipped with a linear map
\begin{align}
\label{eq:inter-form}
Y_{\module}(\ , x) : V\otimes_{\C}\module&\rightarrow \module\db{x}\nonumber\\
a\otimes u&\mapsto  Y_{\module}(a, x)\lu=\sum_{n\in\Z}a_{n}\lu x^{-n-1}
\end{align}
such that the following conditions are satisfied:
\begin{enumerate}
\item $Y_{\module}(\vac,x)=\id_{\module}$.
\item
For $a,b\in V$ and $\lu\in \module$,
\begin{align}
\label{eq:inter-borcherds}
&x_0^{-1}\delta(\dfrac{x_1-x_2}{x_0})Y_{\module}(a,x_1)Y_{\module}(b,x_2)\lu-
x_0^{-1}\delta(\dfrac{x_2-x_1}{-x_0})Y_{\module}(b,x_2)Y_{\module}(a,x_1)\lu\nonumber\\
&=x_1^{-1}\delta(\dfrac{x_2+x_0}{x_1})Y_{\module}(Y(a,x_0)b,x_2)\lu.
\end{align}
\end{enumerate}
\end{definition}
For $i\in\Z$, we define
\begin{align}
	\Z_{< i}&=\{j\in\Z\ |\ j< i\}\mbox{ and }\Z_{> i}=\{j\in\Z\ |\ j> i\}.
\end{align}
For $n\in\C$ and a weak $V$-module $\module$, we define $M_{n}=\{\lu\in V\ |\ \omega_1 \lu=n \lu\}$.
For $a\in V_{n}\ (n\in\C)$, $\wt a$ denotes $n$.
For a vertex algebra $V$ which admits a decomposition $V=\oplus_{n\in\Z}V_n$ and a subset $U$ of a weak $V$-module, we 
define
\begin{align}
\label{eq:OmegaV(U)=BigluinU}
\Omega_{V}(U)&=\Big\{\lu\in U\ \Big|\ 
\begin{array}{l}
a_{i}\lu=0\ \mbox{for all homogeneous }a\in V\\
\mbox{and }i>\wt a-1.
\end{array}\Big\}.
\end{align}
For a vertex algebra $V$ which admits a decomposition $V=\oplus_{n\in\Z}V_n$, a weak $V$-module $\mN$
 is called {\it $\N$-graded} if $N$ admits a decomposition $N=\oplus_{n=0}^{\infty}N(n)$
such that $a_{i}N(n)\subset N(\wt a-i-1+n)$ for all homogeneous $a\in V$, $i\in\Z$, and $n\in\Z_{\geq 0}$, where 
we define $N(n)=0$ for all $n<0$. For a triple of weak $V$-modules $\module, \mN,\mW$,
$\lu\in\module, \lv\in\mW$, and an intertwining operator $I(\ ,x)$ from $\module\times \mW$ to $\mN$, 
we write the expansion of $I(u,x)v$ by
\begin{align}
I(\lu,x)\lv
&=\sum_{i\in\C}\lu_i\lv x^{-i-1}
\in \mN\{x\}.
\end{align}
In this paper, we consider only the case that the image of $I(\mbox{ },x)$ is contained in
$\mN\db{x}$,
namely $I(\ ,x) : \module\times \mW\rightarrow \mN\db{x}$. 
For a subset $X$ of $\mW$,
\begin{align}
M\cdot X\mbox{ denotes }\Span_{\C}\{a_{i}\lu\ |\ a\in \module, i\in\Z, \lu\in X\}\subset N. 
\end{align}
For an intertwining operator $I(\mbox{ },x) : \module\times \mW\rightarrow \mN\db{x}$,
$\lu\in\module$, and $\lv\in \mW$, we define  $\epsilon_{I}(\lu,\lv)=\epsilon(\lu,\lv)\in\Z\cup\{-\infty\}$ by
\begin{align}
\label{eqn:max-vanish}
\lu_{\epsilon_{I}(\lu,\lv)}\lv&\neq 0\mbox{ and }\lu_{i}\lv
=0\mbox{ for all }i>\epsilon_{I}(\lu,\lv)
\end{align}
if $I(\lu,x)\lv\neq 0$ and $\epsilon_{I}(\lu,\lv)=-\infty$
if $I(\lu,x)\lv= 0$.
If $\module$ is irreducible, then $\epsilon_{I}(\lu,\lv)\in\Z$ by \cite[Proposition 11.9]{DL}.
For a weak module $(\module,Y_{\module})$, we write $\epsilon_{\module}=\epsilon_{Y_{\module}}$ for simplicity.

We recall the {\it Zhu algebra} $A(V)$ of a vertex operator algebra $V$ from \cite[Section 2]{Z1996}.
For homogeneous $a\in V$ and $b\in V$, we define
\begin{align}
\label{eq:zhu-ideal-multi}
a\circ b&=\sum_{i=0}^{\infty}\binom{\wt a}{i}a_{i-2}b\in V
\end{align}
and 
\begin{align}
\label{eq:zhu-bimodule-left}
a*b&=\sum_{i=0}^{\infty}\binom{\wt a}{i}a_{i-1}b\in V.
\end{align}
We extend \eqref{eq:zhu-ideal-multi} and \eqref{eq:zhu-bimodule-left} for any $a\in V$ by linearity.
We also define
$O(V)=\Span_{\C}\{a\circ b\ |\ a,b\in V\}$.
Then, the quotient space
\begin{align}
\label{eq:zhu-bimodule}
A(V)&=M/O(V)
\end{align}
called the {\em Zhu algebra} associated to $V$, is an associative $\C$-algebra with multiplication  
\eqref{eq:zhu-bimodule-left} by \cite[Theorem 2.1.1]{Z1996}.
In \cite[Theorem 2.2.1]{Z1996}, Zhu shows that
for a vertex operator algebra $V$ there is a one to one correspondence between the set of all isomorphism classes of irreducible $\N$-graded weak $V$-modules 
and that of irreducible modules for the Zhu algebra associated to $V$.

For a finite dimensional vector space $\hei$ equipped with a non-degenerate symmetric bilinear form
$\langle \mbox{ }, \mbox{ }\rangle$,
$M(1)$ denotes {\it the Heisenberg vertex operator algebra} constructed from $\fh$ (cf. \cite[Section 6.3]{LL}).
As a vector space, $M(1)$ is the symmetric algebra of $\fh\otimes \C[t^{-1}]$.
For $\alpha\in \hei$ and $n\in\Z_{<0}$, $\alpha(n)$ denotes $\alpha\otimes t^{n}\in\fh\otimes \C[t^{-1}]$. 
The conformal vector of $M(1)$ is given by
\begin{align}
\label{eq:conformal-vector}
\omega&=\dfrac{1}{2}\sum_{i=1}^{\dim\fh}h_i(-1)h_i^{\prime}(-1)\vac
\end{align}
where $\{h_1,\ldots,h_{\dim\fh}\}$ is a basis of $\fh$ and
$\{h_1^{\prime},\ldots,h_{\dim\fh}^{\prime}\}$ is its dual basis.
For $\beta\in\fh$, $M(1,\beta)$ denotes the irreducible $M(1)$-module generated by 
the vector $e^{\beta}$ such that $(\alpha(-1)\vac)_{0}e^{\beta}=\langle\alpha,\beta\rangle e^{\beta}$
and $(\alpha(-1)\vac)_{n}e^{\beta}=0$ for  all $\alpha\in\fh$ and $n\in\Z_{>0}$.

Let  $\hat{\lattice}$ be the canonical central extension of $\lattice$
by the cyclic group $\langle\kappa\rangle$ of order $2$ with  the commutator map
$c(\alpha,\beta)=\kappa^{\langle\alpha,\beta\rangle}$ for $\alpha,\beta\in\lattice$:
\begin{align}
	0\rightarrow \langle\kappa\rangle\overset{}{\rightarrow} \hat{\lattice}\overset{-}{\rightarrow} \lattice\rightarrow 0.
\end{align}
Taking $M(1)$ for $\fh:=\C\otimes_{\Z}\lattice$,
we define $V_{\lambda+\lattice}:=\oplus_{\beta\in\lambda+\lattice}M(1,\beta)$
for $\lambda+\lattice\in \lattice^{\perp}/\lattice$.
Then, $V_{\lattice}$ admits a unique vertex algebra structure compatible with the action of $M(1)$ and
is called {\it the lattice vertex algebra} (cf. \cite[Section 6.4]{LL}). 
For each $\lambda+\lattice\in\lattice^{\perp}/\lattice$
the $M(1)$-module
$V_{\lambda+\lattice}$ is an irreducible weak $V_{\lattice}$-module which admits the following decomposition:
\begin{align}
V_{\lambda+\lattice}&=\bigoplus_{n\in\langle\lambda,\lambda\rangle/2+\Z}(V_{\lambda+\lattice})_{n} \mbox{ where }
(V_{\lambda+\lattice})_{n}=\{a\in V_{\lambda+\lattice}\ |\ \omega_{1}a=na\}.
\end{align}
It is shown in \cite[Theorem 3.1]{Dong1993} that
$\{V_{\lambda+\lattice}\ |\ \lambda+\lattice\in \lattice^{\perp}/L\}$ is a complete set of representatives of equivalence classes of the irreducible weak $V_{\lattice}$-modules
where $L^{\perp}$ is the dual lattice of $\lattice$, and in \cite[Theorem 3.16]{DLM1997} that every weak $V_{\lattice}$-module is completely reducible. 
Note that if $\lattice$ is positive definite, then $\dim_{\C}(V_{\lambda+\lattice})_{n}<+\infty$
for all $n\in \lambda+\lattice$
and $(V_{\lambda+\lattice})_{\langle\lambda,\lambda\rangle/2+i}=0$ for sufficiently small $i\in\Z$.
If $\lattice$ is not positive definite, then  
\begin{align}
\label{eq:dimC(Vlambda+lattice)n=+infty}
\dim_{\C}(V_{\lambda+\lattice})_{n}=+\infty
\end{align}
for all $n\in \langle\lambda,\lambda\rangle/2+\Z$, which implies that $V_{\lambda+\lattice}$ is not a $V_{\lattice}$-module.
The $-1$-isometry of $\lattice$ induces an automorphism of $\hat{\lattice}$ of order $2$ and an automorphism of $V_{\lattice}$ of order $2$.
By abuse of notation we denote these automorphisms by the same symbol $\theta$.
For a weak $V_{\lattice}$-module $\module$,
we define a weak $V_{\lattice}$-module $(\module\circ \theta,Y_{\module\circ \theta})$
by $\module\circ\theta=\module$ and 
\begin{align}
Y_{\module\circ \theta}(a,x)&=Y_{\module}(\theta(a),x)
\end{align}
for  $a\in V_{\lattice}$.
Then
$V_{\lambda+\lattice}\circ\theta\cong V_{-\lambda+\lattice}$
for $\lambda\in \lattice^{\perp}$.
Thus, for $\lambda\in \lattice^{\perp}$ with $2\lambda\in\lattice$
we define
\begin{align}
V_{\lambda+\lattice}^{\pm}&=\{u\in V_{\lambda+\lattice}\ |\ \theta(u)=\pm u\}.
\end{align}
For $\alpha\in \fh$, we define
\begin{align}
E(\alpha)&=e^{\alpha}+\theta(e^{\alpha}).
\end{align}
Let $\alpha\in\lattice\setminus\{0\}$. By \cite[(6.4.62)]{LL}, we have
\begin{align}
\label{eq:epsilonVlambda+lattice(E(alpha)-1}
\textcolor{black}{\epsilon_{V_{\lambda+\lattice}}(E(\alpha),e^{\lambda})}&=|\langle\alpha,\lambda\rangle|-1
\end{align}
for $\lambda\in\lattice^{\perp}\setminus\{0, \pm\alpha\}$
and 
\begin{align}
\label{eq:epsilonVlambda+lattice(E(alpha)-2}
\textcolor{black}{\epsilon_{V_{\lambda+\lattice}^{\pm}}(E(\alpha),e^{\lambda}+\theta(e^{\lambda}))}&=|\langle\alpha,\lambda\rangle|-1.
\end{align}
for $\lambda\in\lattice^{\perp}\setminus\{0,\pm\alpha\}$ with $2\lambda\in\lattice$.

Set a submodule $\subL=\{\theta(a) a^{-1}\ |\ a\in\hat{\lattice}\}$ of $\hat{\lattice}$.
The vector space $M(1)(\theta)$ denotes the $\theta$-twisted $M(1)$-module,
$T_{\chi}$ denotes the irreducible $\hat{L}/\subL$-module associated to a
central character $\chi$ such that $\chi(\kappa)=-1$, and
$V_{\lattice}^{T_{\chi}}$ denotes the irreducible $\theta$-twisted $V_{\lattice}$-module
associated to $\chi$ (cf. \cite[Section 9]{FLM} and \cite{Dong1994}).
We define 
\begin{align}
M(1)(\theta)^{\pm}&=\{u\in M(1)(\theta)\ |\ \theta(u)=\pm u\},\nonumber\\
(V_{\lattice}^{T_{\chi}})^{\pm}&=\{u\in V_{\lattice}^{T_{\chi}}\ |\ \theta(u)=\pm u\}.
\end{align}
It is shown in \cite[Theorem 4.5]{DN1999-1} and \cite[Theorem 6.2.2]{DN2001} that 
the following is a complete set of representatives of equivalence classes of the irreducible irreducible $M(1)^{+}$-module:
\begin{align}
M(1)^{\pm}, M(1,\lambda)\cong M(1,-\lambda)\mbox{ with }\lambda\in\fh\setminus\{0\}, M(1)(\theta)^{\pm}.
\end{align}
It is also shown in \cite[Thorem 1.1]{Tanabe2019} that for a non-degenerate even lattice $\lattice$ of finite rank, the following is a complete set of representatives of equivalence classes of the irreducible weak $V_{\lattice}^{+}$-modules:
\begin{enumerate}
\item
$V_{\lambda+\lattice}^{\pm}$, $\lambda+\lattice\in \lattice^{\perp}/\lattice$ with $2\lambda\in \lattice$,   
\item
$V_{\lambda+\lattice}\cong V_{-\lambda+\lattice}$, $\lambda+\lattice\in \lattice^{\perp}/\lattice$ with $2\lambda\not\in \lattice$,   
\item
$V_{\lattice}^{T_{\chi},\pm}$ for any irreducible $\hat{\lattice}/K$-module $T_{\chi}$ with central character $\chi$.
\end{enumerate}
Let $h^{[1]},\ldots,h^{[\rankL]}$ be an orthonormal basis of $\fh$.
For $i=1,\ldots, \rankL$, 
we define
\begin{align}
\label{eq:def-oega-i-H-i}
\omega^{[i]}&=\dfrac{1}{2}h^{[i]}(-1)^2,\nonumber\\
\omega&=\omega^{[1]}+\cdots+\omega^{[\rankL]},\nonumber\\
\Har^{[i]}&=\dfrac{1}{3}h^{[i]}(-3)h^{[i]}(-1)\vac-\dfrac{1}{3}h^{[i]}(-2)^2\vac,\nonumber\\
\Har^{\langle 6\rangle, [i]}&=\dfrac{1}{5}h^{[i]}(-5)h^{[i]}(-1)\vac-\dfrac{13}{10}h^{[i]}(-4)h^{[i]}(-2)\vac+\dfrac{11}{10}(h(-3)^{[i]})^2\vac.
\end{align}
We recall the following notation and some results from \textcolor{black}{\cite[Sections 4 and 5]{DN2001}}:
for any pair of distinct elements $i,j\in\{1,\ldots,\rankL\}$ and $l,m\in\Z_{>0}$, we define
\begin{align}
\nS_{ij}(l,m)&=h^{[i]}(-l)h^{[j]}(-m),\nonumber\\
E^{u}_{ij}&=5\nS_{ij}(1,2)+25\nS_{ij}(1,3)+36\nS_{ij}(1,4)+16\nS_{ij}(1,5),\nonumber\\
E^{t}_{ij}&=-16\nS_{ij}(1,2)+145\nS_{ij}(1,3)+19\nS_{ij}(1,4)+8\nS_{ij}(1,5),\nonumber\\
\Lambda_{ij}
&=45\nS_{ij}(1,2)+190\nS_{ij}(1,3)+240\nS_{ij}(1,4)+96\nS_{ij}(1,5).
\end{align}
It follows from \cite[Proposition 5.3.14]{DN2001} that in $A(M(1)^{+})$, $A^{u}=\oplus_{i,j}\C E^{u}_{ij}$ and 
$A^{t}=\oplus_{i,j}\C E^{t}_{ij}$ are two-sided ideals, each of  which is isomorphic to the $\rankL\times \rankL$ matrix algebra
and $A^{u}A^{t}=A^{t}A^{u}=0$. 
By \cite[Proposition 5.3.15]{DN2001},
$A(M(1)^{+})/(A^{u}+A^{t})$ is a commutative algebra generated by
the images of $\omega^{[i]},\Har^{[i]}$ and $\Lambda_{jk}$ where
$i=1,\ldots,\rankL$ and $j,k\in \{1,\ldots,\rankL\}$ with $j\neq k$.

It is well known that $M(1)^{+}$ is strongly generated by
$\omega^{[i]},\Har^{[i]}$, and $S_{lm}(1,r)$ $(1\leq i\leq \rankL, 1\leq m<l\leq \rankL, r=1,2,3)$
in the sense of \cite[p.111]{Kac1998}.
For a coset $\Gamma\in \lattice^{\perp}/\lattice$ with $\Gamma\neq\lattice$ and  $2\Gamma\subset\lattice$,
we take a subset $S_{\Gamma}$ of $\Gamma$ so that
$\Gamma=S_{\Gamma}\cup (-S_{\Gamma})$ and $S_{\Gamma}\cap (-S_{\Gamma})=\varnothing$
where $-S_{\Gamma}=\{-\mu\ |\ \mu\in S_{\Gamma}\}$.
We also take a subset $S_{\lattice}$ of $\lattice$ so that
 $\lattice\setminus\{0\}=S_{\lattice}\cup (-S_{\lattice})$ and $S_{\lattice}\cap (-S_{\lattice})=\varnothing$.
Then,
\begin{align}
\label{VL+module-decomposition-M1+}
V_{\lattice}^{+}&\cong M(1)^{+}\oplus \bigoplus_{\alpha\in S_{\lattice}}M(1,\alpha),\nonumber\\
V_{\lattice}^{-}&\cong M(1)^{-}\oplus \bigoplus_{\alpha\in S_{\lattice}}M(1,\alpha),\nonumber\\
V_{\lambda+\lattice}&\cong \bigoplus_{\mu\in \lambda+\lattice}M(1,\mu)\mbox{ for $\lambda\in\lattice^{\perp}\setminus\lattice$ with $2\lambda\not\in\lattice$},\nonumber\\
V_{\lambda+\lattice}^{\pm}&\cong \bigoplus_{\mu\in S_{\lambda+\lattice}}M(1,\mu)\mbox{ for $\lambda\in\lattice^{\perp}\setminus\lattice$ with $2\lambda\in\lattice$},\nonumber\\
V_{\lattice}^{T_{\chi},\pm}&\cong M(1)(\theta)^{\pm}\otimes_{\C} T_{\chi}
\end{align}
as $M(1)^{+}$-modules.

The rest of this section we assume $\rank\lattice=\dim_{\C}\fh=1$.
Let $\wh=h^{[1]}$ be an orthonormal basis of $\fh$.
We write $\omega=\omega^{[1]},\Har=\Har^{[1]},\Harsix=\Har^{\langle 6\rangle, [1]}$ for simplicity.
In \cite[(4.4)]{Abe2005}, $\Har$ and $\Harsix$ are denoted by $\Har^{4}$ and $\Har^{6}$ respectively.
We recall the following formulas from \cite[(4.6)--(4.9)]{Abe2005}: for $m\in\Z$,
\begin{align}
[\Harfour_{3},\omega_{m+1}]&=-3m \Harfour_{m+3}-\frac{m^2(m+1)}{2}\omega_{m+1},
\label{eq;[H^4(0),omega(m)]-1}\\
[\Harfour_{3},[\Harfour_{3},\omega_{m+1}]]&=9m^2 \Harsix_{m+5}+\frac{3m^3(5m+9)}{4}\Harfour_{m+3}\nonumber\\
&\quad{}+
\frac{m^3(3m^3+7m^2+3m-3)}{10}\omega_{m+1},
\label{eq;[H^4(0),omega(m)]-2}\\
[\Harsix_{5},\omega_{m+1}]&=-5m \Harsix_{m+5}-\frac{15m^2(m+1)}{4}\Harfour_{m+3}\nonumber \\
&\quad{}-\frac{m^2(m+1)(2m^2+2m-1)}{6}\omega_{m+1},
\label{eq;[H^4(0),omega(m)]-3}\\
[\Harfour_{4},\omega_{m+1}]&=-(3m-1)\Harfour_{m+4}-\frac{m(m+1)(3m+1)}{6}\omega_{m+2}.
\label{eq;[H^4(0),omega(m)]-31}
\end{align}
By \eqref{eq;[H^4(0),omega(m)]-1}--\eqref{eq;[H^4(0),omega(m)]-3}, 
\begin{align}
\label{eq;[H^4(0),omega(m)]-4}
\omega_{m+1}&=
\frac{-5}{m^3}[\Harfour_{3},\omega_{m+1}]+\frac{5}{m^6}[\Harfour_{3},[\Harfour_{3},
\omega_{m+1}]]+\frac{9}{m^5}[\Harsix_{5},\omega_{m+1}]\nonumber\\
&=
\frac{-5}{m^3}(\Harfour_{3}\omega_{m+1}-\omega_{m+1}\Harfour_{3})\nonumber\\
&\quad{}+\frac{5}{m^6}((\Harfour_{3})^2\omega_{m+1}-2\Harfour_{3}\omega_{m+1}\Harfour_{3}+\omega_{m+1}(\Harfour_{3})^2)\nonumber\\
&\quad{}+\frac{9}{m^5}(\Harsix_{5}\omega_{m+1}-\omega_{m+1}\Harsix_{5}).
\end{align}

Let $\lattice=\Z \alpha$. We write $\mn=\langle\alpha,\alpha\rangle\in 2\Z\setminus\{0\}$ and $\ExB=\ExB(\alpha)$ for simplicity.
As stated in \cite[[(3.27), (3.32), (3.33), and (3.35)]{Tanabe2019}, the following elements in $V_{\lattice}^{+}$ are zero:
\begin{align}
\label{eq:-2376omega-2omega-2omega-1vac}
\sv^{(8),H}&=
-2376\omega_{-2 } \omega_{-2 } \omega_{-1 } \vac
+3168\omega_{-3 } \omega_{-1 } \omega_{-1 } \vac
-6256\omega_{-3 } \omega_{-3 } \vac
-11799\omega_{-4 } \omega_{-2 } \vac
\nonumber\\&\quad{}
+30456\omega_{-5 } \omega_{-1 } \vac
+2310\omega_{-7 } \vac
-9504\omega_{-1 } \omega_{-1 } \Har_{-1 } \vac
-6024\omega_{-3 } \Har_{-1 } \vac
\nonumber\\&\quad{}
-13419\omega_{-2 } \Har_{-2 } \vac
-6516\omega_{-1 } \Har_{-3 } \vac
+11868\Har_{-5 } \vac+5040\Har_{-1 }^2 \vac,\\
\label{eq:big(2(mn-2)(-27 + 54mn - 44mn2+ 40mn3)omega-3-1}
Q^{(4)}
&=
2(\mn-2)(-27 + 54 \mn - 44 \mn^2 + 40 \mn^3)\omega_{-3 }\ExB
\nonumber\\&\quad{}
-12\mn (\mn-2) (-3 + 4 \mn)\omega_{-1 }^2\ExB
\nonumber\\&\quad{}
-6 \mn( \mn-2 )  (-9 + 2 \mn) (-1 + 2 \mn)\Har_{-1 }\ExB
\nonumber\\&\quad{}
+(-72\mn^3-96\mn^2+210\mn-90)\omega_{0 } \omega_{-2 }\ExB
\nonumber\\&\quad{}
+(120\mn^2-48\mn+36)\omega_{0 }^2\omega_{-1 }\ExB
\nonumber\\&\quad{}
+(-48\mn-9)\omega_{0 }^4\ExB,\\
Q^{(5,1)}&=
3 (\mn-2) (10 \mn^2-29 \mn+32) (10 \mn^2-4 \mn+3)\omega_{-4 } E
\nonumber\\&\quad{}
-12 \mn (3 \mn-4) (10 \mn^2-4 \mn+3)\omega_{-2 } \omega_{-1 } E
\nonumber\\&\quad{}
-3 (\mn-8) (\mn-2) (2 \mn-1) (10 \mn^2-4 \mn+3)\Har_{-2 } E
\nonumber\\&\quad{}
+8 (2 \mn-7) (15 \mn^3-22 \mn^2+8 \mn-6)\omega_{0 } \omega_{-3 } E
\nonumber\\&\quad{}
+24 \mn^2 (8 \mn-9)\omega_{0 } \omega_{-1 } \omega_{-1 } E
\nonumber\\&\quad{}
-12 (\mn-2) (2 \mn-1) (6 \mn^2-5 \mn+6)\omega_{0 } \Har_{-1 } E
\nonumber\\&\quad{}
-6 (2 \mn^3-32 \mn^2+29 \mn+12)\omega_{0 } \omega_{0 } \omega_{-2 } E
\nonumber\\&\quad{}
-6 (8 \mn-9)\omega_{0 } \omega_{0 } \omega_{0 } \omega_{0 } \omega_{0 } E,
\end{align}
\begin{align}
Q^{(6)}&=
 2 (3696 \mn^8-22564 \mn^7+66284 \mn^6-84937 \mn^5+56207 \mn^4
\nonumber\\&\qquad{}
-91528 \mn^3+11774 \mn^2+29190 \mn-13500)\omega_{-5 } E
\nonumber\\&\quad{}
-4 \mn (352 \mn^6+2152 \mn^5-8282 \mn^4+7951 \mn^3-11696 \mn^2\nonumber\\&\qquad{}
+6304 \mn-1542)\omega_{-3 } \omega_{-1 } E
\nonumber\\&\quad{}
-3 \mn (1584 \mn^6-5572 \mn^5+6456 \mn^4-6877 \mn^3+5214 \mn^2\nonumber\\&\qquad{}
-3040 \mn+642)\omega_{-2 } \omega_{-2 } E
\nonumber\\&\quad{}
+720 \mn^3 (\mn-2) (4 \mn-1)\omega_{-1 } \omega_{-1 } \omega_{-1 } E
\nonumber\\&\quad{}
-24 \mn (\mn-2) (2 \mn-1) (44 \mn^4-98 \mn^3+157 \mn^2-88 \mn+48)\omega_{-1 } \Har_{-1 } E
\nonumber\\&\quad{}
-3 (\mn-2) (2 \mn-25) (2 \mn-1)^2 (44 \mn^4-13 \mn^3+62 \mn^2-48 \mn+18)\Har_{-3 } E
\nonumber\\&\quad{}
+3 (1760 \mn^7-9382 \mn^6+1391 \mn^5+28130 \mn^4-14380 \mn^3\nonumber\\&\qquad{}
+29762 \mn^2-25851 \mn+7650)\omega_{0 } \omega_{-4 } E
\nonumber\\&\quad{}
+12 \mn (352 \mn^5-1459 \mn^4+2396 \mn^3-2894 \mn^2+1254 \mn-225)\omega_{0 } \omega_{-2 } \omega_{-1 } E
\nonumber\\&\quad{}
-3 (\mn-2) (2 \mn-1) (352 \mn^5+101 \mn^4+86 \mn^3-614 \mn^2+804 \mn-225)\omega_{0 } \Har_{-2 } E
\nonumber\\&\quad{}
+12 (88 \mn^6+1104 \mn^5-4136 \mn^4+3714 \mn^3-3944 \mn^2+2670 \mn-675)\omega_{0 } \omega_{0 } \omega_{-3 } E
\nonumber\\&\quad{}
-6 (352 \mn^5-1099 \mn^4+686 \mn^3-689 \mn^2+804 \mn-225)\omega_{0 } \omega_{0 } \omega_{0 } \omega_{-2 } E
\nonumber\\&\quad{}
-90 (\mn-2) (4 \mn-1)\omega_{0 } \omega_{0 } \omega_{0 } \omega_{0 } \omega_{0 } \omega_{0 } E.
\end{align}

The following result will be used in Lemma \ref{lemma:Ext1V(Vlattice-,V\lattice+}.
\begin{lemma}
\label{lemma:M(1)-minus-expansion}
Let $\lattice=\Z \alpha$ be an even lattice of rank one 
such that $\langle\alpha,\alpha\rangle\not\in\{0,2\}$
and
$\module$ a weak $V_{\lattice}^{+}$-module such that $V_{\lattice}^{+}$ is a submodule of $\module$.
Let $u$ be a non-zero element of $\module$ such that  
$\wideHar_{4}\lu, \wideomega_{2}\lu\in \C\vac\subset V_{\lattice}^{+}$, and $\wideHar_{4+i}\lu=\wideomega_{2+i}\lu=0$ for all $i\in\Z_{>0}$.
\begin{enumerate}
\item If  $\epsilon_{\module}(\ExB(\alpha),\lu)\geq 0$, $\wideomega_{1}\lu=\wideHar_{3}\lu=\lu$, and $\Har_{4}\lu=0$, then $\epsilon_{\module}(\omega,\lu)=1$
and $\epsilon_{\module}(\ExB(\alpha),\lu)=0$.
\item If  $\epsilon_{\module}(\ExB(\alpha),\lu)\geq -1$, $\wideomega_{1}\lu=\wideHar_{3}\lu=0$, and  $\Har_{4}\lu=(-1/3)\omega_{2}\lu$, then $\epsilon_{\module}(\omega,\lu)\leq 0$ and $\epsilon_{\module}(\ExB(\alpha),\lu)=-1$.
\end{enumerate}
\end{lemma}

\begin{proof}
We write $\mn=\langle\alpha,\alpha\rangle$ and $\ExB=\ExB(\alpha)$ for simplicity.
Let $\lE$ be an integer such that $\lE\geq \epsilon_{\module}(\ExB(\alpha),\lu)$.
Using \cite[Lemma 2.2]{Tanabe2019}, we expand $Q^{(4)}_{\lE+4}\lu,Q^{(5,1)}_{\lE+5}\lu$, and $Q^{(6)}_{\lE+6}\lu$ 
so that the resulting expressions are linear combinations of elements of the form 
\begin{align}
	a^{(1)}_{i_1}\cdots a^{(l)}_{i_l}\ExB_{m}b^{(1)}_{j_1}\cdots b^{(n)}_{j_n}\lu
\end{align}
where $l,n\in\Z_{\geq 0}$, $m\in\Z$, and 
\begin{align}
	(a^{(1)},i_1),\ldots,(a^{(l)},i_l)&\in\{(\wideomega,k)\ |\ k\leq 0\}\cup\{(\wideHar,k)\ |\ k\leq 2\},\nonumber\\
	(b^{(1)},j_1),\ldots,(b^{(n)},j_n)&\in\{(\wideomega,k)\ |\ k\geq 1\}\cup\{(\wideHar,k)\ |\ k\geq 3\}.
\end{align}
Then, we have the following results:
\begin{align}
\label{eq:-(lE+1)2((16mn+3)lE2-1}
0&=
-(\lE+1)^2 ((16 \mn+3) \lE^2+(-16 \mn^2+36 \mn+12) \lE\nonumber\\&\qquad{}+4 \mn^3-18 \mn^2+14 \mn+12)\ExB_{\lE } \lu
\nonumber\\&\quad{}
+2 ((20 \mn^2-8 \mn+6) \lE^2+(-16 \mn^3+44 \mn^2-10 \mn+12) \lE\nonumber\\&\qquad{}+4 \mn^4-18 \mn^3+18 \mn^2+\mn+6)\ExB_{\lE } \wideomega_{1 } \lu
\nonumber\\&\quad{}
+2 ((20 \mn^2-8 \mn+6) \lE^2+(-28 \mn^3+28 \mn^2+25 \mn-3) \lE\nonumber\\&\qquad{}+20 \mn^4-70 \mn^3+42 \mn^2+6 \mn)\ExB_{\lE-1 } \wideomega_{2 } \lu
\nonumber\\&\quad{}
-4 \mn (\mn-2) (4 \mn-3)\ExB_{\lE } \wideomega_{1 } \wideomega_{1 } \lu
\nonumber\\&\quad{}
-8 \mn (\mn-2) (4 \mn-3)\ExB_{\lE-1 } \wideomega_{1 } \wideomega_{2 } \lu
\nonumber\\&\quad{}
-2 \mn (\mn-2) (2 \mn-9) (2 \mn-1)\ExB_{\lE } \wideHar_{3 } \lu
\nonumber\\&\quad{}
-2 \mn (\mn-2) (2 \mn-9) (2 \mn-1)\ExB_{\lE-1 } \wideHar_{4 } \lu,\\
\label{eq:-(lE+1)2((16mn+3)lE2-2}0&=
2 (\lE+1)^2 ((8 \mn-9) \lE^3+(72 \mn^2+44 \mn-45) \lE^2\nonumber\\
&\qquad{}+(-78 \mn^3+166 \mn^2+115 \mn-78) \lE+20 \mn^4-88 \mn^3+52 \mn^2+112 \mn-48)\ExB_{\lE } \lu
\nonumber\\&\quad{}
-4 ((102 \mn^3-64 \mn^2+43 \mn-12) \lE^2+(-78 \mn^4+226 \mn^3-77 \mn^2+71 \mn-24) \lE\nonumber\\
&\qquad{}+20 \mn^5-88 \mn^4+82 \mn^3+20 \mn^2+22 \mn-12)\ExB_{\lE } \wideomega_{1 } \lu
\nonumber\\&\quad{}
-2 ((202 \mn^3-96 \mn^2+57 \mn-36) \lE^2+(-240 \mn^4+134 \mn^3+560 \mn^2-295 \mn+12) \lE\nonumber\\
&\qquad{}+220 \mn^5-818 \mn^4+558 \mn^3+117 \mn^2-102 \mn)\ExB_{\lE-1 } \wideomega_{2 } \lu
\nonumber\\&\quad{}
-8 \mn ((8 \mn^2-9 \mn) \lE-30 \mn^3+92 \mn^2-70 \mn+12)\ExB_{\lE } \wideomega_{1 } \wideomega_{1 } \lu
\nonumber\\&\quad{}
-4 \mn ((32 \mn^2-36 \mn) \lE-150 \mn^3+420 \mn^2-305 \mn+60)\ExB_{\lE-1 } \wideomega_{1 } \wideomega_{2 } \lu
\nonumber\\&\quad{}
+4 (\mn-2) (2 \mn-1) ((6 \mn^2-5 \mn+6) \lE+10 \mn^3-54 \mn^2+10 \mn+6)\ExB_{\lE }\wideHar_{3 } \lu
\nonumber\\&\quad{}
+(\mn-2) (2 \mn-1) ((24 \mn^2-20 \mn+24) \lE+50 \mn^3-300 \mn^2+75 \mn)\ExB_{\lE-1 }\wideHar_{4 } \lu,
\end{align}
\begin{align}
\label{eq:-(lE+1)2((16mn+3)lE2-3}
0&=
-30 (\lE+1)^2 ((12 \mn^2-27 \mn+6) \lE^4+(228 \mn^2-513 \mn+114) \lE^3
\nonumber\\&\qquad{}
+(704 \mn^5-569 \mn^4+994 \mn^3+335 \mn^2-2316 \mn+501) \lE^2
\nonumber\\&\qquad{}+(-704 \mn^6+2246 \mn^5-2098 \mn^4+2602 \mn^3+1314 \mn^2-4422 \mn+864) \lE
\nonumber\\&\qquad{}+176 \mn^7-932 \mn^6+1496 \mn^5-1175 \mn^4+410 \mn^3+2656 \mn^2-3198 \mn+540)\ExB_{\lE } \lu
\nonumber\\&\quad{}
-60 ((176 \mn^5-455 \mn^4+514 \mn^3-349 \mn^2+276 \mn-45) \lE^3
\nonumber\\&\qquad{}
+(-1056 \mn^6+2355 \mn^5-4527 \mn^4+4854 \mn^3-2877 \mn^2+1791 \mn-270) \lE^2
\nonumber\\&\qquad{}
+(704 \mn^7-2950 \mn^6+5142 \mn^5-8303 \mn^4+6886 \mn^3-3486 \mn^2+2574 \mn-405) \lE
\nonumber\\&\qquad{}
-176 \mn^8+932 \mn^7-1848 \mn^6+2609 \mn^5-3033 \mn^4+1085 \mn^3
\nonumber\\&\qquad\quad{}-352 \mn^2+990 \mn-180)\ExB_{\lE } \omega_{1 } \lu
\nonumber\\&\quad{}
-6 ((2112 \mn^5-5649 \mn^4+5826 \mn^3-4179 \mn^2+3564 \mn-675) \lE^3
\nonumber\\&\qquad{}
+(-11792 \mn^6+21956 \mn^5-34675 \mn^4+31030 \mn^3-11613 \mn^2+10944 \mn-2025) \lE^2
\nonumber\\&\qquad{}
+(10384 \mn^7-18258 \mn^6+8401 \mn^5+8859 \mn^4-76550 \mn^3+86952 \mn^2-19887 \mn+900) \lE
\nonumber\\&\qquad{}-10560 \mn^8+50640 \mn^7-91020 \mn^6+126690 \mn^5-114150 \mn^4
\nonumber\\&\qquad\quad
+22110 \mn^3+21600 \mn^2-4770 \mn)\ExB_{\lE-1 } \omega_{2 } \lu
\nonumber\\&\quad{}
+120 \mn ((176 \mn^5-455 \mn^4+694 \mn^3-754 \mn^2+366 \mn-45) \lE
\nonumber\\&\qquad{}
-352 \mn^6+1434 \mn^5-2623 \mn^4+3729 \mn^3-3523 \mn^2+1524 \mn-180)\ExB_{\lE } \omega_{1 } \omega_{1 } \lu
\nonumber\\&\quad{}
+12 \mn ((3872 \mn^5-10559 \mn^4+16276 \mn^3-17974 \mn^2+8574 \mn-1125) \lE
\nonumber\\&\qquad{}
-10032 \mn^6+36036 \mn^5-61980 \mn^4+88110 \mn^3-81438 \mn^2+34164 \mn-4050)\ExB_{\lE-1 } \omega_{1 } \omega_{2 } \lu
\nonumber\\&\quad{}
+720 \mn^3 (\mn-2) (4 \mn-1)\ExB_{\lE } \omega_{1 } \omega_{1 } \omega_{1 } \lu
\nonumber\\&\quad{}
+2160 \mn^3 (\mn-2) (4 \mn-1)\ExB_{\lE-1 } \omega_{1 } \omega_{1 } \omega_{2 } \lu
\nonumber\\&\quad{}
-6 (\mn-2) (2 \mn-1) ((880 \mn^5-190 \mn^4+800 \mn^3-1580 \mn^2+1800 \mn-450) \lE
\nonumber\\&\qquad{}
+880 \mn^6-6596 \mn^5+9372 \mn^4-14533 \mn^3+11462 \mn^2-72 \mn-450)\ExB_{\lE } \Har_{3 } \lu
\nonumber\\&\quad{}
-24 \mn (\mn-2) (2 \mn-1) (44 \mn^4-98 \mn^3+157 \mn^2-88 \mn+48)\ExB_{\lE } \Har_{3 } \omega_{1 } \lu
\nonumber\\&\quad{}
-24 \mn (\mn-2) (2 \mn-1) (44 \mn^4-98 \mn^3+157 \mn^2-88 \mn+48)\ExB_{\lE-1 } \Har_{3 } \omega_{2 } \lu
\nonumber\\&\quad{}
-3 (\mn-2) (2 \mn-1) ((2112 \mn^5-279 \mn^4+1686 \mn^3-3774 \mn^2+4404 \mn-1125) \lE
\nonumber\\&\qquad{}
+2640 \mn^6-22780 \mn^5+29470 \mn^4-47255 \mn^3+39830 \mn^2-6000 \mn)\ExB_{\lE-1 } \Har_{4 } \lu
\nonumber\\&\quad{}
-24 \mn (\mn-2) (2 \mn-1) (44 \mn^4-98 \mn^3+157 \mn^2-88 \mn+48)\ExB_{\lE-1 } \Har_{4 } \omega_{1 } \lu.
\end{align}
We note that $\wideomega_{i}\wideomega_{2}\lu=\wideomega_{i}\wideHar_{4}\lu=\wideHar_{i}\wideomega_{2}\lu=\wideHar_{i}\wideHar_{4}\lu=\ExB_{i}\wideomega_{2}\lu= \ExB_{i}\wideHar_{4}\lu=0$ for all $i\geq 0$ in \eqref{eq:-(lE+1)2((16mn+3)lE2-1}--\eqref{eq:-(lE+1)2((16mn+3)lE2-3} since $\wideHar_{4}\lu, \wideomega_{2}\lu\in \C\vac$.
\begin{enumerate}
\item
Substituting $\Har_{3}\lu=\omega_{1}\lu=\lu$ and $\Har_{4}\lu=0$ into 
\eqref{eq:-(lE+1)2((16mn+3)lE2-1}--\eqref{eq:-(lE+1)2((16mn+3)lE2-3} 
and deleting the terms including $\ExB_{\lE-1}\omega_{2}\lu$ from the obtained relations, we have
the following results:
\begin{align}
\label{eq:lE2(10mn2-4mn+3)-1}
0&=\lE^2
(10 \mn^2-4 \mn+3) \nonumber\\&\quad{}\times ((32 \mn-36) \lE^5+(-80 \mn^2+281 \mn-198) \lE^4
\nonumber\\&\qquad{}
+(24 \mn^3-376 \mn^2+783 \mn-378) \lE^3
\nonumber\\&\qquad{}
+(-12 \mn^4+176 \mn^3-673 \mn^2+887 \mn-294) \lE^2
\nonumber\\&\qquad{}
+(24 \mn^5+16 \mn^4-370 \mn^3+660 \mn^2-367 \mn+78) \lE
\nonumber\\&\qquad{}
-8 \mn^6-40 \mn^5+498 \mn^4-1424 \mn^3+1697 \mn^2-864 \mn+156)\ExB_{\lE}\lu,\\
\label{eq:lE2(10mn2-4mn+3)-2}
0&=
\lE^2 ((2400 \mn^4-6360 \mn^3+4080 \mn^2-2100 \mn+360) \lE^6
\nonumber\\&\qquad{}
+(33792 \mn^6-87408 \mn^5+137589 \mn^4-189186 \mn^3+124737 \mn^2-41898 \mn+5355) \lE^5
\nonumber\\&\qquad{}
+(-81664 \mn^7+382200 \mn^6-726064 \mn^5+1055205 \mn^4-1153306 \mn^3
\nonumber\\&\quad\qquad{}
+669765 \mn^2-193386 \mn+22095) \lE^4
\nonumber\\&\qquad{}
+(25344 \mn^8-412844 \mn^7+1292964 \mn^6-2146889 \mn^5+
2965056 \mn^4-2780396 \mn^3
\nonumber\\&\quad\qquad{}
+1427091 \mn^2-369111 \mn+39375) \lE^3
\nonumber\\&\qquad{}
+(-9152 \mn^9+180504 \mn^8-874824 \mn^7+1952142 \mn^6-2951142 \mn^5\nonumber\\&\quad\qquad{}
+3684897 \mn^4-2957890 \mn^3+1327089 \mn^2-306300 \mn+30465) \lE^2
\nonumber\\&\qquad{}
+(20416 \mn^{10}+4648 \mn^9-377744 \mn^8+1165462 \mn^7-2048798 \mn^6+2665009 \mn^5
\nonumber\\&\quad\qquad{}
-2403075 \mn^4+1390194 \mn^3-493932 \mn^2+103029 \mn-9630) \lE
\nonumber\\&\qquad{}
-7040 \mn^{11}-36640 \mn^{10}+546080 \mn^9-2091064 \mn^8+4414496 \mn^7
\nonumber\\&\quad\qquad{}
-6655612 \mn^6+7729798 \mn^5-6378462 \mn^4+3419178 \mn^3-1123344 \mn^2
\nonumber\\&\quad\qquad{}
+210006 \mn-17820)\ExB_{\lE}\lu.
\end{align}
Deleting $\lE^i\ (i\geq 3)$ from \eqref{eq:lE2(10mn2-4mn+3)-1} and \eqref{eq:lE2(10mn2-4mn+3)-2} (cf. (1) in the proof of \textcolor{black}{\cite[Lemma 3.10]{Tanabe2019}}), we have 
\begin{align}
\label{eq:lE2mn2(mn-2)(2mn-1)2(8mn-9}
0&=
\lE^2\mn
(\mn-2)
(2 \mn-1)
(8 \mn-9)
g_1g_2g_3g_4g_5\ExB_{\lE}\lu
\end{align}
where
\begin{align*}
g_1&=720896 \mn^8-6533120 \mn^7+12732160 \mn^6-11571376 \mn^5+8753247 \mn^4
\nonumber\\&\quad\qquad{}
-6117402 \mn^3+2934828 \mn^2-718146 \mn+40095,
\nonumber\\
g_2&=
30976 \mn^10-93632 \mn^9-274896 \mn^8+676496 \mn^7+70580 \mn^6
\nonumber\\&\quad\qquad{}
-1376964 \mn^5+1569114 \mn^4-766098 \mn^3+138753 \mn^2+14580 \mn-6561,
\nonumber\\
g_3&=32480690176 \mn^{17}-686708228096 \mn^{16}+4100113563648 \mn^{15}
\nonumber\\&\qquad{}
-9261356843008 \mn^{14}+4721613180928 \mn^{13}+27252555600512 \mn^{12}
\nonumber\\&\qquad{}
-89074476796752 \mn^{11}+153500461862476 \mn^{10}-188119215355208 \mn^9
\nonumber\\&\qquad{}
+182750873232189 \mn^8-146333976903441 \mn^7+95994434529360 \mn^6
\nonumber\\&\qquad{}
-50207935079160 \mn^5+20230310418021 \mn^4-6034639211379 \mn^3
\nonumber\\&\qquad{}
+1264660375698 \mn^2-169751870910 \mn+11394160650,
\nonumber\\
g_4&=
108295298266169344 \mn^{28}-1578632220535422976 \mn^{27}
\nonumber\\&\qquad{}
+9311723577440993280 \mn^{26}
-30020826118265765888 \mn^{25}
\nonumber\\&\qquad{}
+56480859583533809664 \mn^{24}-34201181968036986880 \mn^{23}
\nonumber\\&\qquad{}
-166134072751850102784 \mn^{22}+653988138655346800640 \mn^{21}
\nonumber\\&\qquad{}
-1269321655065859079168 \mn^{20}
+1420723402232025004160 \mn^{19}
\nonumber\\&\qquad{}
-291714635577902883936 \mn^{18}-2582683778801669449520 \mn^{17}
\nonumber\\&\qquad{}
+6770894754722207547920 \mn^{16}-10944187113221235636592 \mn^{15}
\nonumber\\&\qquad{}
+13555805386036866935018 \mn^{14}
-13737333902075174510823 \mn^{13}
\nonumber\\&\qquad{}
+11720517307272109891506 \mn^{12}-8535382247957070808665 \mn^{11}
\nonumber\\&\qquad{}
+5336352269983480520232 \mn^{10}-2867703024491554846995 \mn^9
\nonumber\\&\qquad{}
+1322914878412254530550 \mn^8-522431105033231973729 \mn^7
\nonumber\\&\qquad{}
+175773803489479604430 \mn^6-49918422180208562604 \mn^5
\nonumber\\&\qquad{}
+11741259740661392544 \mn^4-2205270148985139708 \mn^3
\nonumber\\&\qquad{}
+309462836760861120 \mn^2-28630525595816700 \mn
\nonumber\\&\qquad{}+1296455338665000
\end{align*}
\begin{align}
g_5&=
125419598061634584576 \mn^{38}-4058775327682313846784 \mn^{37}
\nonumber\\&\quad{}
+61364532083457777467392 \mn^{36}-582936719229084397731840 \mn^{35}
\nonumber\\&\quad{}
+3943379872911033271844864 \mn^{34}-20371782223780152285790208 \mn^{33}
\nonumber\\&\quad{}
+83996702697827014071894016 \mn^{32}-284517625089600742756231168 \mn^{31}
\nonumber\\&\quad{}
+805663391079536900830277632 \mn^{30}-1920640734900916387237049344 \mn^{29}
\nonumber\\&\quad{}
+3836954501857572677445864064 \mn^{28}-6271979192897114647393521856 \mn^{27}
\nonumber\\&\quad{}
+7760578552733394021465806368 \mn^{26}-4960446924702123953479190320 \mn^{25}
\nonumber\\&\quad{}-7465071488664393767860076264 \mn^{24}+35708801005589262962946399572 \mn^{23}
\nonumber\\&\quad{}-84222727411601389378170057962 \mn^{22}+152434039009084653430665673211 \mn^{21}
\nonumber\\&\quad{}-232437295292831894456601157082 \mn^{20}+309491421641544381910932790025 \mn^{19}
\nonumber\\&\quad{}-366054612786171099332876352444 \mn^{18}+388035250642103571764606350307 \mn^{17}
\nonumber\\&\quad{}-370337448344968087579269096586 \mn^{16}+318820688670520334196524737033 \mn^{15}
\nonumber\\&\quad{}-247602958017955809443453495754 \mn^{14}+173233782350652213882420588060 \mn^{13}
\nonumber\\&\quad{}-108885985952016051266803373628 \mn^{12}+61222945466874004618327931304 \mn^{11}
\nonumber\\&\quad{}-30611663342976087375389167464 \mn^{10}+13506540937664381785106978544 \mn^9
\nonumber\\&\quad{}-5208033211875215795239060032 \mn^8+1733872407863047422573508704 \mn^7
\nonumber\\&\quad{}-490817666374303669993349712 \mn^6+115795112960087840103771888 \mn^5
\nonumber\\&\quad{}-22152274459759464382843776 \mn^4+3301481575103941278100080 \mn^3
\nonumber\\&\quad{}-359562883864819986252000 \mn^2+25437319335136373184000 \mn
\nonumber\\&\quad{}-875923032140153280000.
\end{align}
Substituting $\lE=\epsilon_{\module}(\ExB(\alpha),\lu)$ into \eqref{eq:lE2mn2(mn-2)(2mn-1)2(8mn-9},
we have $\epsilon_{\module}(\ExB(\alpha),\lu)=0$ since $g_i\neq 0$ for all $\mn\in\Z$ and $i=1,\ldots,5$.
Substituting $\lE=0$ into \eqref{eq:-(lE+1)2((16mn+3)lE2-1}, we have
\begin{align}
0&=\mn(10\mn^3-35\mn^2+21\mn+3)\ExB(\alpha)_{-1}\omega_{2}\lu.
\end{align}
Since $\omega_{2}\lu\in \C\vac$, we have $\omega_{2}\lu=0$ and hence $\epsilon_{\module}(\omega,\lu)=1$.
\item 
The same argument as above shows that if $\wideomega_{1}\lu=\wideHar_{3}\lu=0$ and $\Har_{4}\lu=(-1/3)\omega_{2}\lu$, then $\epsilon_{\module}(\ExB(\alpha),\lu)=-1$.
Substituting $t=0$, $\Har_{3}\lu=\omega_{1}\lu=0$, and $\Har_{4}\lu=(-1/3)\omega_{2}\lu$ into 
\eqref{eq:-(lE+1)2((16mn+3)lE2-1}
and deleting the terms including $\ExB_{0}\lu$ from the obtained relation, we have
$\ExB(\alpha)_{-1}\omega_{2}\lu=0$. Since $\omega_{2}\lu\in\C \vac$, we have $\omega_{2}\lu=0$ and hence \textcolor{black}{$\epsilon_{\module}(\omega,\lu)\leq 0$}.
\end{enumerate}
\end{proof}

\section{Proof of Theorem \ref{theorem:main}}
\label{section:main-result}
In this section we shall show Theorem \ref{theorem:main}.
Let $\lattice$ be a non-degenerate even lattice of finite rank $\rankL$.
Since \cite[Lemma 7.3]{Tanabe2019} shows that 
every weak $V_{\lattice}^{+}$-module has an irreducible weak $V_{\lattice}^{+}$-module,
it is enough to show that $\Ext^{1}_{V_{\lattice}^{+}}(\module,\mW)=0$
for any pair of irreducible weak $V_{\lattice}^{+}$-modules $\module$ and $\mW$ (cf. the last half of the proof of \cite[Theorem 2.8]{Abe2005}).

\begin{lemma}
\label{lemma:generate-completely}
Let $\module$ be a weak $V_{\lattice}^{+}$-modules and
$\mK$ an $M(1)^{+}$-submodule of $\module$
which is isomorphic to $M(1)^{\pm }$, $M(1)(\theta)^{\pm}$, or $M(1,\lambda)$ for some 
$\lambda\in\lattice^{\perp}\setminus\lattice$.
Then, $V_{\lattice}^{+}\cdot \mK$ is a completely reducible weak $V_{\lattice}^{+}$-module.
\end{lemma}
\begin{proof}
If $\mK\cong M(1)(\theta)^{\pm}$, then  $V_{\lattice}^{+}\cdot K$ is a $V_{\lattice}^{+}$-module
by \cite[Lemma 7.2 (1)]{Tanabe2019} and hence is completely reducible by \cite[Theorem 3.16]{Yamsk2009}.
If $\mK\cong M(1)^{+}$, then $V_{\lattice}^{+}\cdot K\cong V_{\lattice}^{+}$ as $V_{\lattice}^{+}$-modules
by \cite[Proposition 4.7.7]{LL}.
If $\mK\cong M(1)^{-}$, then $V_{\lattice}^{-}\cdot K\cong V_{\lattice}^{-}$ as $V_{\lattice}^{+}$-modules
by \cite[Lemma 7.1]{Tanabe2019}.
If $\mK\cong M(1,\lambda)$ for some $\lambda\in \lattice^{\perp}\setminus\lattice$,
then the proof of \cite[Lemma 7.3]{Tanabe2019} shows that 
$V_{\lattice}^{+}\cdot K\cong V_{\lambda+\lattice}$ if $2\lambda\not\in \lattice$
and 
$V_{\lattice}^{+}\cdot K\cong V_{\lambda+\lattice}^{\pm}$ or $V_{\lambda+\lattice}^{+}\oplus
V_{\lambda+\lattice}^{-}$ if $2\lambda\in \lattice$. Thus, $V_{\lattice}^{+}\cdot K$ is a completely reducible weak $V_{\lattice}^{+}$-module
in each case. The proof is complete.
\end{proof}

\begin{lemma}
\label{lemma:extm1vl-standard}
Let $\module$ and $\mW$ be two irreducible weak $V_{\lattice}^{+}$-modules.
We set a $M(1)^{+}$-submodule $\mK$ of $\module$ by
\begin{align}
\mK=\left\{
\begin{array}{ll}
M(1)^{+},&\mbox{if }\module\cong V_{\lattice}^{+},\\
M(1)^{-},&\mbox{if }\module\cong V_{\lattice}^{-},\\
\mbox{any irreducible $M(1)^{+}$-submodule of $\module$},&\mbox{otherwise}.
\end{array}
\right.
\end{align}
If $\Ext^{1}_{M(1)^{+}}(\mK,U)=0$ for any $M(1)^{+}$-submodule $U$ of $\mW$,
then $\Ext^{1}_{V_{\lattice}^{+}}(\module,\mW)=0$.
\end{lemma}
\begin{proof}
Let $0\rightarrow W\rightarrow N\overset{\pi}{\rightarrow} M\rightarrow 0$ be an exact sequence 
of weak $V_{\lattice}^{+}$-modules.
Since any weak $V_{\lattice}^{+}$-module is a direct sum of 
irreducible $M(1)^{+}$-modules, 
it follows from \cite[Lemma 2.6]{Abe2005} that $\Ext^{1}_{M(1)^{+}}(\mK,\mW)=0$
and hence we can take an $M(1)^{+}$-submodule of $\mN$
which is isomorphic to $K$ and intersects with $\mW$ trivially.
We denote again  by $\mK$ this $M(1)^{+}$-submodule of $\mN$.
Since $K\cap \mW=0$ and $\module$ is an irreducible weak $V_{\lattice}^{+}$-module, we have $\mN=\mW+(V_{\lattice}^{+}\cdot K)$.
By Lemma \ref{lemma:generate-completely}, $V_{\lattice}^{+}\cdot \mK$ is a completely reducible weak $V_{\lattice}^{+}$-module
and hence
$\mN\cong \module\oplus \mW$ as $V_{\lattice}^{+}$-modules.
\end{proof}

\begin{corollary}
\label{corollary:many-ext=0}
If a pair $(M,W)$ of irreducible $V_{\lattice}^{+}$-modules satisfies one of the following conditions,
then $\Ext^{1}_{V_{\lattice}^{+}}(M,W)=0$.
\begin{enumerate}
\item $M\cong V_{\lattice}^{+}$ and $W\not\cong V_{\lattice}^{-}$.
\item $M\cong V_{\lattice}^{-}$ and $W\not\cong V_{\lattice}^{+}$.
\item $M\cong V_{\lambda+\lattice}$ with $\lambda\in\lattice^{\perp}$ and $2\lambda\not\in\lattice$
 and $W\not\cong V_{\lambda+\lattice}$.
\item $M\cong V_{\lambda+\lattice}^{+}$ with $\lambda\in\lattice^{\perp}$ and $2\lambda\in\lattice$
 and $W\not\cong V_{\lambda+\lattice}^{\pm}$.
\item $M\cong V_{\lambda+\lattice}^{-}$ with $\lambda\in\lattice^{\perp}$ and $2\lambda\in\lattice$
 and $W\not\cong V_{\lambda+\lattice}^{\pm}$.
\item $M\cong V_{\lattice}^{T_{\chi},\pm}$ where $\chi$ is a central character for $\hat{\lattice}/\mK$
with $\chi(\kappa)=-1$.
\end{enumerate}
\end{corollary}
\begin{proof}
The result follows from \eqref{VL+module-decomposition-M1+}, Lemma \ref{lemma:extm1vl-standard}, 
\cite[Proposition 5.8]{Tanabe2019}, and \cite[Theorem 3.16]{Yamsk2009}
\end{proof}
Let $h^{[1]},\ldots,h^{[\rankL]}$ be an orthonormal basis of $\fh$.
For a weak $V_{\lattice}^{+}$-module $\mN$ and $\zeta=(\zeta^{[i]})_{i=1}^{\rankL},\xi=(\xi^{[i]})_{i=1}^{\rankL}$, $\rho=(\rho^{[i]})_{i=1}^{\rankL}\in \C^{\rankL}$,
we define a generalized eigenspace
\begin{align}
N_{\zeta,\xi,\rho}&=\Big\{u\in\mN\ \Big|\ 
\begin{array}{l}
\mbox{For $i=1,\ldots,\rankL$ there exists $j\in\Z_{>0}$ such that}\\
(\omega^{[i]}_{1}-\zeta^{[i]})^j\lu=
(\Har^{[i]}_{3}-\xi^{[i]})^j\lu=
(\Har^{\langle 6\rangle, [i]}_{5}-\rho^{[i]})^j\lu=0
\end{array}\Big\}.
\end{align}

Let $0\rightarrow W\rightarrow N\overset{\pi}{\rightarrow} M\rightarrow 0$ be an exact sequence of weak $V_{\lattice}^{+}$-modules
where $\mW$ and $\module$ are irreducible weak $V_{\lattice}^{+}$-modules.
For all $\lu\in N_{\zeta,\xi,\rho}$ and $a,b\in 
\{\omega^{[i]}_{1}-\zeta^{[i]},\Har^{[i]}_{3}-\xi^{[i]},
\Har^{\langle 6\rangle, [i]}_{5}-\rho^{[i]}\ |\ i=1,\ldots,\rankL\}$, we have
\begin{align}
	\label{eq:aluinM}
a\lu&\in M \mbox{ and  }ab\lu=0
\end{align}
since an arbitrary irreducible weak $V_{\lattice}^{+}$-module is a completely reducible module for the commutative $\C$-algebra
generated by $\{
\omega^{[i]}_{1},\Har^{[i]}_{3},\Har^{\langle 6\rangle, [i]}_{5}\ |\ i=1,\ldots,\rankL\}$.
For $\alpha\in \fh$ with $\alpha\in \C h^{[1]}$ and
$\lu\in \mN_{\zeta,\xi,\rho}$, we have
\begin{align}
\label{eq:(omega[1]1-zeta[1]+fracmn}
&(\omega^{[1]}_{1}-\zeta^{[1]}-\frac{\langle\alpha,\alpha\rangle}{2}+\epsilon(\ExB(\alpha),\lu)+1)^2\ExB(\alpha)_{\epsilon(\ExB(\alpha),\lu)}\lu\nonumber\\
&=
\ExB(\alpha)_{\epsilon(\ExB(\alpha),\lu)}(\omega^{[1]}_{1}-\zeta^{[1]})^2\lu=0.
\end{align}

\begin{lemma}
\label{lemma:lattice-basis-change}
Let $\lattice$ be a non-degenerate integral lattice of finite rank $\rankL$.
and $\lambda$ a non-zero element of $\Q\otimes_{\Z}L$.
Then, there exists a sublattice $\oplus_{i=1}^{\rankL}\Z\alpha^{[i]}$ of $\lattice$ 
such that 
$\langle\alpha_i,\alpha_i\rangle\not\in\{0,2\}$ for all $i=1,\ldots,\rankL$,
$\langle\alpha_i,\alpha_j\rangle=0$ for all pairs of distinct elements $i,j\in\{1,\ldots,\rankL\}$,
and
$\langle\alpha_i,\lambda\rangle\neq 0$ for all $i=1,\ldots,\rankL$.
\end{lemma}
\begin{proof}
We take a sublattice $\Gamma=\oplus_{i=1}^{\rankL}\Z\alpha^{[i]}$ of $\lattice$ 
so that
$\langle\alpha^{[i]},\alpha^{[i]}\rangle\neq 0$ for all $i=1,\ldots,\rankL$ and
$\langle\alpha^{[i]},\alpha^{[j]}\rangle=0$ for any distinct pair of elements $i,j\in\{1,\ldots,\rankL\}$.
Since $\lattice$ is non-degenerate, we may assume $\langle\alpha^{[1]},\lambda\rangle\neq 0$.
If $\langle\alpha^{[2]},\lambda\rangle=0$,
then writing $p=\langle\alpha^{[1]},\alpha^{[1]}\rangle$ and $q=\langle\alpha^{[2]},\alpha^{[2]}\rangle$
for simplicity, 
we take a pair of non-zero integers $x$ and $y$ so that $p x^2+q y^2\not\in\{0,\pm 1,\pm 2\}$
and set 
\begin{align}
\beta^{[1]}&:=x\alpha^{[1]}+y\alpha^{[2]},
\quad
\beta^{[2]}:=-yq\alpha^{[1]}+px\alpha^{[2]}\in \lattice.
\end{align}
Then, 
\begin{align}
\langle\beta^{[1]},\beta^{[1]}\rangle&=p x^2+q y^2\neq 0,&
\langle\beta^{[2]},\beta^{[2]}\rangle&=pq(px^2+q y^2)\neq 0,\nonumber\\
\langle\beta^{[1]},\beta^{[2]}\rangle&=0,\nonumber\\
\langle\beta^{[i]},\alpha^{[j]}\rangle&=0\ (i=1,2,j=3,\ldots,\rankL),\nonumber\\
\langle\beta^{[1]},\lambda\rangle&=x\langle\alpha^{[1]},\lambda\rangle\neq 0,&
\langle\beta^{[2]},\lambda\rangle&=-yq\langle\alpha^{[1]},\lambda\rangle\neq 0,
\end{align}
and
\begin{align}
(\alpha^{[1]},\alpha^{[2]})&=\frac{1}{(p x^2+q y^2)}(\beta^{[1]},\beta^{[2]})
\begin{pmatrix}
px&yq\\
-y&x
\end{pmatrix}.
\end{align}
We replace $\alpha^{[1]}$ and $\alpha^{[2]}$ by $\beta^{[1]}$ and $\beta^{[2]}$ respectively.
Repeating this procedure, we have the result.
\end{proof}
\begin{remark}
The reason why we avoid the case that $\langle\alpha^{[i]},\alpha^{[i]}\rangle=2$ in Lemma \ref{lemma:lattice-basis-change}
is that when $\langle\alpha^{[i]},\alpha^{[i]}\rangle=2$, we need to change the generators of $V_{\Z\alpha^{[i]}}^{+}$ and hence the commutation relations in the proof of Lemmas \ref{lemma:2lambda(not)inlattice} and 
\ref{lemma:Ext1V(Vlattice-,V\lattice+} below.
\end{remark}
\begin{lemma}
\label{lemma:2lambda(not)inlattice}
For $\lambda\in\lattice^{\perp}$ with $2\lambda\not\in\lattice$,
we have
$\Ext^{1}_{V_{\lattice}^{+}}(V_{\lambda+\lattice},V_{\lambda+\lattice})=0$.
For $\lambda\in\lattice^{\perp}\setminus\lattice$ with $2\lambda\in\lattice$,
we have
$\Ext^{1}_{V_{\lattice}^{+}}(V_{\lambda+\lattice}^{+},V_{\lambda+\lattice}^{+})=
\Ext^{1}_{V_{\lattice}^{+}}(V_{\lambda+\lattice}^{-},V_{\lambda+\lattice}^{-})=
\Ext^{1}_{V_{\lattice}^{+}}(V_{\lambda+\lattice}^{+},V_{\lambda+\lattice}^{-})=\Ext^{1}_{V_{\lattice}^{+}}(V_{\lambda+\lattice}^{-},V_{\lambda+\lattice}^{+})=0$.
\end{lemma}
\begin{proof}
Let $\lambda\in\lattice^{\perp}\setminus\lattice$.
By Lemma \ref{lemma:lattice-basis-change}, we can take a sublattice $\oplus_{i=1}^{\rankL}\Z\alpha_i$ of $\lattice$ of rank $\rankL$
such that 
$\langle\alpha_i,\alpha_i\rangle\not\in\{0,2\}$ for all $i=1,\ldots,\rankL$,
$\langle\alpha_i,\alpha_j\rangle=0$ for any distinct pair of elements $i,j\in\{1,\ldots,\rankL\}$,
and
$\langle\alpha_i,\lambda\rangle\neq 0$ for all $i=1,\ldots,\rankL$.
We take an orthonormal basis $h^{[1]},\ldots,h^{[\rankL]}$ of $\fh=\C\otimes_{\Z}\lattice$ defined by
\begin{align}
\label{eq:h[i]=frac1sqrtlanglealpha[i]}
h^{[i]}&=\frac{1}{\sqrt{\langle\alpha^{[i]},\alpha^{[i]}\rangle}}\alpha^{[i]}\quad (i=1,\ldots,\rankL).
\end{align}
We set $\zeta=(\langle\lambda,h^{[1]}\rangle^2/2,\ldots,\langle\lambda,h^{[\rankL]}\rangle^2/2)$ and $\zero=(0,\ldots,0)$.
When $2\lambda\not\in\lattice$,  we set $(\module,\mW)=(V_{\lambda+\lattice}, V_{\lambda+\lattice})$, and
$\lv=e^{\lambda}\in \module$, $\lw=e^{\lambda}\in \mW$.
When $2\lambda\in\lattice$, we set $(\module,\mW)=(V_{\lambda+\lattice}^{\rho},V_{\lambda+\lattice}^{\sigma})$ with $\rho,\sigma\in\{+,-\}$,
and
$\lv=e^{\lambda}\pm\theta(e^{\lambda})\in\module=V_{\lambda+\lattice}^{\pm}$,
$\lw=e^{\lambda}\pm\theta(e^{\lambda})\in\mW=V_{\lambda+\lattice}^{\pm}$.

Let
$0\rightarrow \mW\rightarrow \mN\overset{\pi}{\rightarrow} \module\rightarrow 0$
be  an exact sequence of weak $V_{\lattice}^{+}$-modules.
We take $\lu\in \mN_{\zeta,\zero,\zero}$ such that $\pi(\lu)=\lv$.
We fix $i\in\{1,\ldots,\rankL\}$.
Using a slight modification of the proof of \cite[Lemma 4.8]{Abe2005},
we shall first show that $\omega^{[i]}_{2+j} \lu=\Harfour^{[i]}_{4+j}\lu=0$ for all $j\geq 0$.
It follows from \cite[Proposition 4.3]{Abe2005} that
$\Har_{3}^{[i]}\lu,\Har_{5}^{6, [i]}\lu\in \C \lw\subset \mW$ and hence
\begin{align}
\label{eq:omega[i]jHar3langle}
\textcolor{black}{\omega^{[i]}_{j}\Har_{3}^{[i]}\lu=\omega^{[i]}_{j}\Har_{5}^{\langle 6\rangle, [i]}\lu}=0
\end{align}
for all $j\in\Z_{\geq 2}$.
For $m\in\Z_{\geq 1}$, by \eqref{eq;[H^4(0),omega(m)]-4} and \eqref{eq:omega[i]jHar3langle},
\begin{align}
\label{eq:(-1+frac-5m^3Hlangle 4rangle,[i]3}
	0&=(-1+
	\frac{-5}{m^3}H^{[i]}_{3}+
	\frac{5}{m^6}(H^{[i]}_{3})^2+
	\frac{9}{m^5}\Har^{\langle 6\rangle,[i]}_{5})\omega^{[i]}_{m+1}\lu.
\end{align}
It follows from \cite[Proposition 4.3]{Abe2005} that $\Har^{[i]}_{3}$ acts diagonally on $\mW$
and the eigenvalues of $\Har^{[i]}_{3}$ on $\mW$ are non-negative integers.
Since $0\neq -1-5a+5a^2+9b$ for any pair of $a,b\in \Z_{\geq 0}$ and $\omega^{[i]}_{2}\lu\in \mW$,
it follows from \eqref{eq:(-1+frac-5m^3Hlangle 4rangle,[i]3} with $m=1$ that
$\omega_2^{[i]}\lu=0$.
We have $\Har^{[i]}_{4}\lu=0$ by \eqref{eq;[H^4(0),omega(m)]-1}, 
$\Har^{\langle 6\rangle, [i]}_{6}\lu=0$ by \eqref{eq;[H^4(0),omega(m)]-3},
and $\Har^{[i]}_{5}\lu=(-2/3)\omega_{3}\lu$ by \eqref{eq;[H^4(0),omega(m)]-31}.
By \eqref{eq;[H^4(0),omega(m)]-1} with $m=2$ and \eqref{eq:omega[i]jHar3langle},
\begin{align}
\Harfour^{[i]}_{3}\omega^{[i]}_{3} \lu
&=-6\Har^{[i]}_{5}\lu-6\omega^{[i]}_{3}\lu=-2\omega_{3}^{[i]}\lu.
\end{align}
Since any eigenvalue of 
$\Harfour^{[i]}_{3}$ on $\mW$
is non-negative, we have $\omega^{[i]}_{3} \lu=0$.
It follows from \eqref{eq;[H^4(0),omega(m)]-1} and
\begin{align}
	[\omega^{[i]}_{j},\omega^{[i]}_{k}]&=(i-j)\omega^{[i]}_{j+k-1}+\delta_{j+k-2,0}\dfrac{j(j-1)(j-2)}{12}\label{eq:wiwj}
\end{align}
for all $j,k\in\Z$ that $\omega^{[i]}_{2+j} \lu=\Harfour^{[i]}_{4+j}\lu=0$ for all $j\geq 0$.

We shall show that $\lu$ is a simultaneous eigenvector for $\{\omega^{[i]}_1,\Har^{[i]}_{3}\}_{i=1}^{\rankL}$
and is an element of $\Omega_{M(1)^{+}}(\mN_{\zeta,\zero,\zero})$.
We fix $i=1,\ldots\rankL$.  After renumbering $\alpha_1,\ldots,\alpha_{\rankL}$, we may assume $i=1$
and we write $\alpha=\alpha^{[1]}, \mn=\langle\alpha^{[1]},\alpha^{[1]}\rangle,E=E(\alpha^{[1]})$, and $\lE=\epsilon(E(\alpha^{[1]}),\lu)$ for simplicity. 
We may assume 
\begin{align}
\label{eq:langlelambdaalpharanglegeq0}
\langle\lambda,\alpha\rangle\geq 0.
\end{align}
Since $N/\mW\cong \module$, we have
\begin{align}
\label{eq:lEgeqlanglelambda}
\lE\geq \langle\lambda,\alpha\rangle-1
\end{align}
by \eqref{eq:epsilonVlambda+lattice(E(alpha)-1} and \eqref{eq:epsilonVlambda+lattice(E(alpha)-2}.
In \cite[(3.68)--(3.71)]{Tanabe2019} we have computed $Q^{(4)}_{\lE+4},\lu$, 
$Q^{(5,1)}_{\lE+5},\lu$,
$Q^{(5,2)}_{\lE+5},\lu$, and
$Q^{(6)}_{\lE+6},\lu$ so that the resulting expressions are linear combinations of elements of the form 
\begin{align}
	a^{(1)}_{i_1}\cdots a^{(l)}_{i_l}\ExB_{m}b^{(1)}_{j_1}\cdots b^{(n)}_{j_n}\lu
\end{align}
where $l,n\in\Z_{\geq 0}$, $m\in\Z$, and 
\begin{align}
	(a^{(1)},i_1),\ldots,(a^{(l)},i_l)&\in\{(\omega,k)\ |\ k\leq 1\}\cup\{(\Har,k)\ |\ k\leq 2\},\nonumber\\
	(b^{(1)},j_1),\ldots,(b^{(n)},j_n)&\in\{(\omega,k)\ |\ k\geq 2\}\cup\{(\Har,k)\ |\ k\geq 3\}.
\end{align}

We shall compute $Q^{(4)}_{\lE+4}\lu,Q^{(5,1)}_{\lE+5}\lu, Q^{(5,2)}_{\lE+5}\lu$, and $Q^{(6)}_{\lE+6}\lu$
so that the resulting expressions are linear combinations of elements of the form 
\begin{align}
	a^{(1)}_{i_1}\cdots a^{(l)}_{i_l}\ExB_{m}b^{(1)}_{j_1}\cdots b^{(n)}_{j_n}\lu
\end{align}
where $l,n\in\Z_{\geq 0}$, $m\in\Z$, and 
\begin{align}
	(a^{(1)},i_1),\ldots,(a^{(l)},i_l)&\in\{(\omega,k)\ |\ k\leq 1\}\cup\{(\Har,k)\ |\ k\leq 3\},\nonumber\\
	(b^{(1)},j_1),\ldots,(b^{(n)},j_n)&\in\{(\omega,k)\ |\ k\geq 2\}\cup\{(\Har,k)\ |\ k\geq 4\}.
\end{align}
The results are
\begin{align}
\label{eq:left-H-1}
0&=
((\lE+1-\mn)^2-2\mn\omega^{[1]}_{1})
\nonumber\\
&\quad{}\times\big(-(\mn-2) (\mn-2 \omega^{[1]}_{1}) (4 \mn-3) +2 \mn (8 \mn-11) \lE-(16 \mn+3) \lE^{2}\big)\ExB_{\lE}\lu
\nonumber\\
&\quad{}-2 \mn (\mn-2) (2 \mn-9) (2 \mn-1)\Harfourone_{3}\ExB_{\lE}\lu,\\
\label{eq:left-H-2}
0&=((\lE+1-\mn)^2-2\mn\omega^{[1]}_{1})\nonumber\\
&\quad{}\times \big(2 (\mn-2) (\mn-2 \omega^{[1]}_{1}) (15 \mn^2-16 \mn+3)\nonumber\\&\qquad{}-(118 \mn^3+(-16 \omega^{[1]}_{1}-193) \mn^2+(18 \omega^{[1]}_{1}+35) \mn+6)t\nonumber\\&\qquad{}+(112 \mn^2+6 \mn-21)t^{2}+(8 \mn-9)t^{3}\big)\ExB_{\lE}\lu\nonumber\\
&\quad{}+2 (\mn-2) (2 \mn-1) ((6 \mn^2-5 \mn+6) \lE+10 \mn^3-54 \mn^2+10 \mn+6)\Harfourone_{3}\ExB_{\lE}\lu,\\
\label{eq:left-H-3}
0&=((\lE+1-\mn)^2-2\mn\omega^{[1]}_{1})\nonumber\\
&\quad{}\times\big((\mn-2) (\mn-2 \omega^{[1]}_{1}) (72 \mn^3+44 \mn^2-235 \mn+120)\nonumber\\
&\qquad{}-(284 \mn^4+(-40 \omega^{[1]}_{1}-3) \mn^3+(6 \omega^{[1]}_{1}-1120) \mn^2+(88 \omega^{[1]}_{1}+754) \mn-48 \omega^{[1]}_{1}+60)\lE\nonumber\\
&\qquad{}+(272 \mn^3+410 \mn^2-363 \mn-270)\lE^{2}\nonumber\\&\qquad{}+(16 \mn^2+61 \mn-102)\lE^{3}\big)\ExB_{\lE}\lu\nonumber\\
&\quad{}+2 (\mn-2) (2 \mn-1) ((14 \mn^3+21 \mn^2-74 \mn+60) \lE+24 \mn^4\nonumber\\
&\qquad{}-90 \mn^3-221 \mn^2+220 \mn+60)\Harfourone_{3}\ExB_{\lE}\lu,\\
\label{eq:left-H-4}
0&=
((\lE+1-\mn)^2-2\mn\omega^{[1]}_{1})\nonumber\\
&\quad{}\times\big(
 -3 (\mn-2) (\mn-2 \omega^{[1]}_{1}) (616 \mn^5-1262 \mn^4+(-40 \omega^{[1]}_{1}+1958) \mn^3\nonumber\\
&\qquad\quad{}+(10 \omega^{[1]}_{1}-2397) \mn^2+1232 \mn-150)\nonumber\\
&\qquad+3 (2464 \mn^6+(-704 \omega^{[1]}_{1}-6537) \mn^5+(1618 \omega^{[1]}_{1}+9700) \mn^4\nonumber\\
&\qquad\quad{}+(-2452 \omega^{[1]}_{1}-12590) \mn^3+(2848 \omega^{[1]}_{1}+7496) \mn^2\nonumber\\
&\qquad\quad{}+(-1388 \omega^{[1]}_{1}-623) \mn+150 \omega^{[1]}_{1}-90)\lE\nonumber\\
&\qquad{}-3 (2464 \mn^5-2793 \mn^4+(40 \omega^{[1]}_{1}+4907) \mn^3+(-90 \omega^{[1]}_{1}-3483) \mn^2\nonumber\\
&\qquad\quad{}+(20 \omega^{[1]}_{1}-1252) \mn+385)\lE^{2}\nonumber\\
&\qquad{}-15 (\mn-2) (2 \mn+19) (4\mn-1)\lE^{3}-15 (\mn-2) (4 \mn-1)\lE^{4}\big)\ExB_{\lE}\lu\nonumber\\
&\quad{}-(\mn-2) (2 \mn-1) ((1056 \mn^5-582 \mn^4+1428 \mn^3-1932 \mn^2+1992 \mn-450) \lE\nonumber\\
&\qquad{}+792 \mn^6+(176 \omega^{[1]}_{1}-6224) \mn^5+(-392 \omega^{[1]}_{1}+8666) \mn^4\nonumber\\
&\qquad{}+
(628 \omega^{[1]}_{1}-13729) \mn^3+(-352 \omega^{[1]}_{1}+11014) \mn^2\nonumber\\
&\qquad{}+(192 \omega^{[1]}_{1}+120) \mn-450)\Harfourone_{3}\ExB_{\lE}\lu.
\end{align}

Deleting the terms including $\ExB_{\lE}\Harfourone_{3}\lu$ and $((\lE+1-\mn)^2-2\mn\omega^{[1]}_{1})^2\ExB_{\lE}\lu$ from 
\cite[(3.68)--(3.71)]{Tanabe2019},
we have
\begin{align}
\label{eq:mn2(2mn-9)(2mn-1)(4mn2-12mn+15)}
0&=
\mn^2(2\mn-9)(2\mn-1)(4\mn^2-12\mn+15)(10\mn^2-4\mn+3)\nonumber\\
&\quad{}\times (44\mn^4-13\mn^3+62\mn^2-48\mn+18)\nonumber\\
&\quad{}\times \lE(\lE-\mn+2)(2\lE-\mn+1)(2\lE-\mn+2)(2\lE-\mn+3)\nonumber\\
&\quad{}\times ((\lE+1-\mn)^2-2\mn\omega^{[1]}_{1})\ExB_{\lE}\lu.
\end{align}
We set 
\begin{align}
\tilde{\lu}&:=((\lE+1-\mn)^2-2\mn\omega^{[1]}_{1})\ExB_{\lE}\lu=\ExB_{\lE}((\lE+1)^2-2\mn\omega^{[1]}_{1})\lu
\end{align}
and assume $\tilde{\lu}\neq 0$.
By \eqref{eq:mn2(2mn-9)(2mn-1)(4mn2-12mn+15)}, we have $\lE=0, \mn/2-1,$ or $\mn-2$.
\begin{enumerate}
\item
If $\lE=\mn-2$, then deleting the terms including $\ExB_{\lE}\Harfourone_{3}\lu$ from \cite[(3.68) and (3.69)]{Tanabe2019},
we have $(\omega_1^{[1]}-1)\tilde{\lu}=0$.
Since
$\langle\lambda,\alpha\rangle^2/(2\mn)+\mn/2-\lE-1=1$ by \eqref{eq:(omega[1]1-zeta[1]+fracmn},
we have
$\langle\lambda,\alpha\rangle^2=\mn^2$ and hence $\lE=\mn-2<\langle\lambda,\alpha\rangle-1$
by \eqref{eq:langlelambdaalpharanglegeq0}, which contradicts \eqref{eq:lEgeqlanglelambda}.

\item
If $\lE=\mn/2-1$, then deleting the terms including $\Harfourone_{3}\ExB_{\lE}\lu$ 
from \eqref{eq:left-H-1}--\eqref{eq:left-H-4}, we have
\begin{align}
0&=(16\omega^{[1]}_1-1)
(16\omega^{[1]}_1-9)\tilde{\lu}
\end{align}
and hence there exists a non-zero $\lw\in \langle\omega^{[1]}_{1}\rangle\lu$
such that $(\omega^{[1]}_{1}\lw, \Har^{[1]}_{3}\lw)=
((1/16)\lw,(-1/128)\lw)$ or 
$((9/16)\lw,(15/128)\lw)$ by \eqref{eq:left-H-1}.
This contradicts to \cite[Proposition 4.3]{Abe2005} which shows that the eigenvalues of $\Harfourone_3$ on 
$V_{\lambda+\lattice}$ for any $\lambda\in\lattice^{\perp}$
are non-negative integers.

\item
If $\lE=0$, then deleting the terms including $\Harfourone_{3}\ExB_{\lE}\lu$ 
from \eqref{eq:left-H-1}--\eqref{eq:left-H-4}, we have $\omega^{[1]}_1\tilde{\lu}=(\mn/2)\tilde{\lu}$.
Since $\langle\lambda,\alpha\rangle/(2\mn)+\mn/2-\lE-1=\mn/2$ by \eqref{eq:(omega[1]1-zeta[1]+fracmn}, we have 
$\langle\lambda,\alpha\rangle=2\mn$, which contradicts \eqref{eq:langlelambdaalpharanglegeq0} and \eqref{eq:lEgeqlanglelambda}.
\end{enumerate}
Thus, $\tilde{\lu}=(\omega^{[1]}_1-(\lE+1-\mn)^2/(2\mn))E_{\lE}\lu=E_{\lE}(\omega_1^{[1]}-(\lE+1)^2/(2\mn))\lu=0$
and hence $(\lE+1)^2=\langle\lambda,\alpha\rangle^2$ by \eqref{eq:(omega[1]1-zeta[1]+fracmn},
which leads 
\begin{align}
\label{eq:lElanglelambda,alpharangle-1}
\lE&=\langle\lambda,\alpha\rangle-1
\end{align}
by \eqref{eq:lEgeqlanglelambda}.
Since $(\omega^{[1]}_1-(\lE+1)^2/(2\mn))\lu\in W_{\zeta,\zero,\zero}=\C\lw$,
it follows from \eqref{eq:epsilonVlambda+lattice(E(alpha)-1} and \eqref{eq:epsilonVlambda+lattice(E(alpha)-2}
that $(\omega^{[1]}_1-(\lE+1)^2/(2\mn))\lu=0$.
By \cite[(3.68)]{Tanabe2019}, $\ExB_{\lE}\Harfourone_{3}\lu=0$.
Since $\Harfourone_{3}\lu\in W_{\zeta,\zero,\zero}=\C\lw$,
it follows from \eqref{eq:epsilonVlambda+lattice(E(alpha)-1} and \eqref{eq:epsilonVlambda+lattice(E(alpha)-2}
that $\Harfourone_{3}\lu=0$.
We conclude that $\lu$ is a simultaneous eigenvector for $\{\omega^{[i]}_1,\Har^{[i]}_{3}\}_{i=1}^{\rankL}$. 
The same argument as in the proof (1) of \cite[Lemma 5.5]{Tanabe2019} shows that $\epsilon(S_{ij}(1,1),\lu)\leq 1$ for 
any pair of distinct elements $i,j\in \{1,\ldots,\rankL\}$
and hence $\lu\in \Omega_{M(1)^{+}}(\mN_{\zeta,\zero,\zero})$.

Since $M(1)^{-}$ and $M(1)(\theta)^{-}$ are not $M(1)^{+}$-submodules of $\mN$, $A^{u}\lu=A^{t}\lu=0$
and hence $A(M(1)^{+})\cdot \lu
=\langle
\Lambda_{ij}\ |\ i,j=1,\ldots,\rankL\mbox{ with }i\neq j\rangle\lu$.
By \cite[(6.1.15)]{DN2001}, 
$\Lambda_{ij}^2\lu=4\omega^{[i]}_{1}\omega^{[j]}_{1}\lu=\langle h^{[i]},\lambda\rangle^2\langle h^{[j]},\lambda\rangle^2\lu\neq 0$.
Since $(\Lambda_{ij}-\langle h^{[i]},\lambda\rangle\langle h^{[j]},\lambda\rangle)^2\lu=0$,
we  have $\Lambda_{ij}\lu=\langle h^{[i]},\lambda\rangle\langle h^{[j]},\lambda\rangle\lu$
and hence $\C \lu\cong M(1,\lambda)(0)$ as $A(M(1)^{+})$-modules.
By \cite[Corollary 5.9]{Tanabe2019}, $M(1)^{+}\cdot \lu\cong M(1,\lambda)$ as $M(1)^{+}$-modules.
Now, the result follows from Lemma \ref{lemma:generate-completely}.
\end{proof}

\begin{lemma}
\label{lemma:Ext1V(Vlattice+,V\lattice-}
$\Ext^{1}_{V_{\lattice}^{+}}(V_{\lattice}^{+},V_{\lattice}^{-})=0$.
\end{lemma}
\begin{proof}
Let
$0\rightarrow V_{\lattice}^{-}\rightarrow \mN\overset{\pi}{\rightarrow} V_{\lattice}^{+}\rightarrow 0$
be an exact sequence of weak $V_{\lattice}^{+}$-modules.
Since  $(V_{L}^{+})_{\zero,\zero,\zero}=\C\vac$ and $(V_{L}^{-})_{\zero,\zero,\zero}=0$,
we can take a non-zero $\lu\in \mN_{\zero,\zero,\zero}$ such that $\pi(\lu)=\vac$ and
$\omega_1^{[i]}\lu=\Har^{[i]}_{3}\lu=\Har^{\langle 6\rangle,[i]}_{5}\lu=0$ for all $i=1,\ldots,\rankL$.
The same argument as in the proof of Lemma \ref{lemma:2lambda(not)inlattice} (cf. \eqref{eq:omega[i]jHar3langle}--
\eqref{eq:wiwj}) shows that 
$\omega^{[i]}_{2+j} \lu=\Harfour^{[i]}_{4+j}\lu=0$ for all $i=1,\ldots,\rankL$ and $j\geq 0$.
Since $\sv^{(8),H}_{6}\lu=0$, we have
\begin{align}
\omega^{[i]}_0\lu&=3\Har^{[i]}_{2}\lu.
\end{align}
For $i,j=1,\ldots,\rankL$,
\begin{align}
\omega_{1}^{[j]}\omega_0^{[i]}\lu
&=-[\omega^{[i]}_0,\omega^{[j]}_{1}]\lu
=\delta_{ij}\omega^{[i]}_{0}\lu
\end{align}
and
\begin{align}
	\Har^{[j]}_{3}\omega^{[i]}_0\lu
	&=-[\omega^{[i]}_0,\Har^{[j]}_{3}]\lu
	=\delta_{ij}3\Har^{[i]}_{2}\lu
	=\delta_{ij}\omega^{[i]}_{0}\lu.
\end{align}
It follows from \cite[Proposition 4.3]{Abe2005} that for any $i=1,\ldots,\rankL$
\begin{align}
\{\lv\in V_{\lattice}^{+}\ |\ \omega_{1}^{[j]}\lv=\Har^{[j]}_{3}\lv=\delta_{ij}\lv
\mbox{ for all }j=1,\ldots,\rankL
\}&=\{0\}\mbox{ and }\nonumber\\
\{\lv\in V_{\lattice}^{-}\ |\ \omega_{1}^{[j]}\lv=\Har^{[j]}_{3}\lv=\delta_{ij}\lv
\mbox{ for all }j=1,\ldots,\rankL
\}&\subset \C h^{[i]}(-1)\vac
\end{align}
and hence that $\omega^{[i]}_0\lu\in \C h^{[i]}(-1)\vac\subset V_{\lattice}^{-}$.
We take $\beta\in\fh$ so that $\beta(-1)\vac=\omega_{0}\lu=\sum_{i=1}^{\rankL}\omega^{[i]}_0\lu\in V_{\lattice}^{-}$.
The following argument is a slight modification of a part of the proof of \cite[Proposition 4.6]{Abe2002}.
Since for all $\alpha\in\lattice$ and $n\in\Z$,
\begin{align}
	\label{eq:E(alpha)nh(-1)vac}
E(\alpha)_{n}\beta(-1)\vac
&=-[\beta(-1),E(\alpha)_{n}]\vac+\beta(-1)E(\alpha)_{n}\vac\nonumber\\
&=-(\beta(0)E(\alpha))_{n-1}\vac+\beta(-1)E(\alpha)_{n}\vac\nonumber\\
&=-\langle \alpha,\beta\rangle (e^{\alpha}-\theta (e^{\alpha}))_{n-1}
\vac+\beta(-1)E(\alpha)_{n}\vac,
\end{align}
we have $E(\alpha)_{n}\omega_{0}\lu=0$ for all $n\geq 1$.
Since $E(\alpha)_{n}=(-1/(n+1))[\omega_0,E(\alpha)_{n+1}]$ for all $n\geq 0$,
we have $\epsilon(\ExB(\alpha),\lu)\leq 0$. Moreover, 
since $E(\alpha)_{0}\lu=-[\omega_0,E(\alpha)_{1}]\lu=0$, we have 
$E(\alpha)_{0}\omega_0\lu=\omega_0E(\alpha)_{0}\lu=0$.
Since $\lattice$ is non-degenerate,
it follows from \eqref{eq:E(alpha)nh(-1)vac} with $n=0$ that $\omega_0\lu=0$ and hence
 $V_{\lattice}^{+}\cdot \lu\cong V_{\lattice}^{+}$.
Now, the result follows from Lemma \ref{lemma:generate-completely}.
\end{proof}

\begin{lemma}
\label{lemma:Ext1V(Vlattice-,V\lattice+}
$\Ext^{1}_{V_{\lattice}^{+}}(V_{\lattice}^{-},V_{\lattice}^{+})=0$.
\end{lemma}
\begin{proof}
Let
$0\rightarrow V_{\lattice}^{+}\rightarrow \mN\overset{\pi}{\rightarrow} V_{\lattice}^{-}\rightarrow 0$
be an exact sequence of weak $V_{\lattice}^{+}$-modules.
By Lemma \ref{lemma:lattice-basis-change}, we can take a sublattice $\oplus_{i=1}^{\rankL}\Z\alpha^{[i]}$ of $\lattice$ 
so that
$\langle\alpha^{[i]},\alpha^{[i]}\rangle\not\in\{0,2\}$ for all $i=1,\ldots,\rankL$ and
$\langle\alpha^{[i]},\alpha^{[j]}\rangle=0$ for all pairs of distinct elements $i,j\in\{1,\ldots,\rankL\}$.
We define $p_i:=\langle\alpha^{[i]},\alpha^{[i]}\rangle$ and $h^{[i]}:=\alpha^{[i]}/\sqrt{p_i}$.
For $i=1,\ldots,\rankL$ and $\zeta^{(i)}=(\delta_{ij})_{j=1}^{\rankL}$, 
let $u^{[i]}\in N_{\zeta^{(i)},\zeta^{(i)},\zeta^{(i)}}$ such that 
\begin{align}
\pi(u^{[i]})=\alpha^{[i]}(-1)\vac\in V_{\lattice}^{-}.
\end{align}
Since $(V_{\lattice}^{+})_{\zeta^{(i)},\zeta^{(i)},\zeta^{(i)}}=\zero$,
we have 
\begin{align}
\label{eq:(omega[j]-deltaij)u[i]}
(\omega^{[j]}_{1}-\delta_{ij})u^{[i]}&=
(\Har^{[j]}_{3}-\delta_{ij})u^{[i]}=
(\Har^{\langle 6\rangle, [j]}_{5}-\delta_{ij})u^{[i]}=0
\end{align}
for $j=1,\ldots,\rankL$.
We define
\begin{align}
U&=\sum_{i=1}^{\rankL}\C \lu^{[i]}.
\end{align} 
Since $\Ext^{1}_{M(1)^{+}}(M(1)^{-},M(1,\alpha))=0$ for all $\alpha\in\fh\setminus\{0\}$ by \cite[Proposition 5.8]{Tanabe2019}
and $M(1)^{+}_{1}=0$, the vector space $U$ is an $A(M(1)^{+})$-module  which is isomorphic to $M(1)^{-}_{1}=M(1)^{-}(0)$.
For any pair of distinct element $i,j\in\{1,\ldots,\rankL\}$ and $l,m\in\Z_{>0}$, we set 
\begin{align}
\widetilde{S}_{ij}(l,m)&=\alpha^{[i]}(-l)\alpha^{[j]}(-m),
\end{align}
which is a non-zero scalar multiple of ${S}_{ij}(l,m)$.
In order to show $U\subset\Omega_{M(1)^{+}}(\mN)$, it is convenient to use $\widetilde{S}_{ij}(l,m)$ instead of ${S}_{ij}(l,m)$.
We fix $i\in\{1,\ldots,\rankL\}$.
For $k=1,\ldots,\rankL$, since $\omega^{[k]}_{2}\lu^{[i]}\in \C \vac$, we have $\Har^{[k]}_{3}\omega^{[k]}_{2}\lu^{[i]}=0$. 
For $j\in\{1,\ldots,\rankL\}$ with $j\neq i$, by \eqref{eq:(omega[j]-deltaij)u[i]} and \eqref{eq;[H^4(0),omega(m)]-1} with $m=1$,
\begin{align}
-\omega^{[i]}_{2}\lu^{[i]}&=[\Har^{[i]}_{3},\omega^{[i]}_{2}]\lu^{[i]}=-\textcolor{black}{3}\Har^{[i]}_{4}\lu^{[i]}-\omega^{[i]}_{2}\lu^{[i]},\nonumber\\
0&=[\Har^{[j]}_{3},\omega^{[j]}_{2}]\lu^{[i]}=-\textcolor{black}{3}\Har^{[j]}_{4}\lu^{[i]}-\omega^{[j]}_{2}\lu^{[i]}
\end{align}
and hence $\Har^{[i]}_{4}\lu^{[i]}=0$ and $\textcolor{black}{3}\Har^{[j]}_{4}\lu^{[i]}=-\omega^{[j]}_{2}\lu^{[i]}$. 
By Lemma \ref{lemma:M(1)-minus-expansion},
we have $\epsilon_{\module}(\omega^{[i]},\lu^{[i]})=1, \epsilon_{\module}(\ExB(\alpha^{[i]}),\lu^{[i]})=0$ 
 and $\epsilon_{\module}(\omega^{[j]},\lu^{[i]})\leq 0$, $\epsilon_{\module}(\ExB(\alpha^{[j]}),\lu^{[i]})=-1$ for $j\neq i$.

We shall show that $\epsilon_{\module}(\widetilde{S}_{ij}(1,1),\lu)\leq 1$ for all non-zero $\lu\in U$ and all pairs of distinct elements $i,j\in\{1,\ldots,\rankL\}$.
We may assume $(i,j)=(2,1)$.
We take a pair of non-zero integers $x$ and $y$ so that $p_1 x^2+p_2 y^2\not\in\{0,\pm 1,\pm 2\}$
and define the sequence $\tilde{\alpha}^{[1]},\tilde{\alpha}^{[2]},\ldots,\tilde{\alpha}^{[\rankL]}\in\lattice$ by 
\begin{align}
\tilde{\alpha}^{[1]}&=x\alpha^{[1]}+y\alpha^{[2]},\nonumber\\
\tilde{\alpha}^{[2]}&=-yp_2\alpha^{[1]}+xp_1\alpha^{[2]},\mbox{ and }\nonumber\\
\tilde{\alpha}^{[i]}&=\alpha^{[i]},\ i=3,\ldots,\rankL.
\end{align}
For $i=1,\ldots,\rankL$ and $j=3,\ldots,\rankL$, we also define 
\begin{align}
\tilde{h}^{[i]}&=\frac{\tilde{\alpha}^{[i]}}{\sqrt{\langle \tilde{\alpha}^{[i]},\tilde{\alpha}^{[i]}\rangle}},\nonumber\\
\tilde{\omega}^{[i]}&=\frac{1}{2}\tilde{h}^{[i]}(-1)^2\vac,\nonumber\\
\tilde{\Har}^{[i]}&=\dfrac{1}{3}\tilde{h}^{[i]}(-3)h^{[i]}(-1)\vac-\dfrac{1}{3}\tilde{h}^{[i]}(-2)^2\vac,\nonumber\\
\tilde{u}^{[1]}&=x\lu^{[1]}+y\lu^{[2]},\nonumber\\
\tilde{u}^{[2]}&=-yp_2\lu^{[1]}+xp_1\lu^{[2]},\nonumber\\
\tilde{u}^{[j]}&=\lu^{[j]}.
\end{align}
Then, $\langle \tilde{\alpha}^{[i]},\tilde{\alpha}^{[i]}\rangle\not\in\{0,2\}$ for all $i=1,\ldots,\rankL$,
 $\langle \tilde{\alpha}^{[i]},\tilde{\alpha}^{[j]}\rangle=0$ for all pairs of distinct element $i,j\in\{1,\ldots,\rankL\}$,
$\tilde{\omega}^{[i]}_{1}\tilde{\lu}^{[j]}=\tilde{\Har}^{[i]}_{3}\tilde{\lu}^{[j]}=\delta_{ij}\tilde{\lu}^{[j]}$, and
$U=\sum_{i=1}^{\rankL}\C \tilde{\lu}^{[i]}$.
The same argument as above shows that $\epsilon(\tilde{\omega}^{[i]},\lu)\leq 1$ for all non-zero $\lu\in U$ and $i=1,\ldots,\rankL$.
For a non-zero $\lu\in U$, since $\epsilon(\tilde{\omega}^{[1]},\lu), \epsilon(\omega^{[k]},\lu)\leq 1$ for all $k=1,\ldots,\rankL$, and
\begin{align}
\widetilde{S}_{21}(1,1)&=\frac{1}{xy}(\langle\tilde{\alpha}^{[1]},\tilde{\alpha}^{[1]}\rangle
\tilde{\omega}^{[1]}-\mn_1x^2\omega^{[1]}-\mn_2y^2\omega^{[2]}),
\end{align}
we have $\epsilon(\widetilde{S}_{21}(1,1),\lu)\leq 1$ and hence $U\subset\Omega_{M(1)^{+}}(\mN)$.
Now, the result follows from Lemma \ref{lemma:generate-completely}.
\end{proof}

\begin{proof}[(Proof of Theorem \ref{theorem:main})]
Since \cite[Lemma 7.3]{Tanabe2019} shows that 
every weak $V_{\lattice}^{+}$-module has an irreducible weak $V_{\lattice}^{+}$-module,
it is enough to show that $\Ext^{1}_{V_{\lattice}^{+}}(\module,\mW)=0$
for any pair of irreducible weak $V_{\lattice}^{+}$-modules $\module$ and $\mW$ (cf. the last half of the proof of \cite[Theorem 2.8]{Abe2005}).
The result follows from Corollary \ref{corollary:many-ext=0}, Lemma \ref{lemma:2lambda(not)inlattice},
Lemma \ref{lemma:Ext1V(Vlattice+,V\lattice-}, and 
Lemma \ref{lemma:Ext1V(Vlattice-,V\lattice+}.
\end{proof}

\input{nondeg-rep.bbl}

\bibliographystyle{ijmart}
\end{document}

%% file: nondeg-rep.bbl
\providecommand{\MR}{\relax\ifhmode\unskip\space\fi MR }
\providecommand{\MRhref}[2]{%
  \href{http://www.ams.org/mathscinet-getitem?mr=#1}{#2}
}
\providecommand{\href}[2]{#2}

%% file: nondeg-rep.bbl
\begin{thebibliography}{10}

\bibitem{Abe2002}
T.~Abe, \textsl{The charge conjugation orbifold {$V_{\Z\alpha}^{+}$} is
  rational when $\langle\alpha,\alpha\rangle/2$ is prime}, Int. Math. Res. Not.
  \textbf{2002} (2002), 647--665.

\bibitem{Abe2005}
T.~Abe, \textsl{Rationality of the vertex operator algebra {$V_{L}^{+}$} for a
  positive definite even lattice {$L$}}, Math. Z. \textbf{249} (2005),
  455--484.

\bibitem{ABD2004}
T.~Abe, G.~Buhl and C.~Dong, \textsl{Rationality, regularity, and
  {$C_2$}-cofiniteness}, Trans. Amer. Math. Soc. \textbf{356} (2004),
  3391--3402.

\bibitem{arakawa2015}
T.~Arakawa, \textsl{Rationality of {$W$}-algebras: principal nilpotent cases},
  Ann. of Math. (2) \textbf{182} (2015), 565--604.

\bibitem{B1986}
R.~Borcherds, \textsl{Vertex algebras, {Kac-Moody} algebras, and the
  {Monster}}, Proc. Nat. Acad. Sci. U.S.A. \textbf{83} (1986), 3068--3071.

\bibitem{DVVV1989}
R.~Dijkgraaf, C.~Vafa, E.~Verlinde, and H.~Verlinde, \textsl{The operator
  algebra of orbifold models}, Comm. Math. Phys. \textbf{123} (1989), 485--526.

\bibitem{Dong1993}
C.~Dong, \textsl{Vertex algebras associated with even lattices}, J. Algebra
  \textbf{160} (1993), 245--265.

\bibitem{Dong1994-moonshine}
C.~Dong, \textsl{Representations of the moonshine module vertex operator
  algebra}, Mathematical aspects of conformal and topological field theories
  and quantum groups ({S}outh {H}adley, {MA}, 1992), Contemp. Math. {\bfseries
  175}, 27--36 (1994).

\bibitem{Dong1994}
C.~Dong, \textsl{Twisted modules for vertex algebras associated with even
  lattices}, J. Algebra \textbf{165} (1994), 91--112.

\bibitem{DJL2012}
C.~Dong, C.~Jiang and X.~Lin, \textsl{Rationality of vertex operator algebra
  {$V_{L}^{+}$}: higher rank}, Proc. Lond. Math. Soc. \textbf{104} (2012),
  799--826.

\bibitem{DLaTYY2004}
C.~Dong, C.~H. Lam, K.~Tanabe, H.~Yamada, and K.~Yokoyama, \textsl{{${\mathbb
  Z}_{3}$} symmetry and {$W_3$} algebra in lattice vertex operator algebras},
  Pacific J. Math. \textbf{215} (2004), 245--296.

\bibitem{DL}
C.~Dong and J.~Lepowsky, \textsl{Generalized vertex algebras and relative
  vertex operators}, Progress in Mathematics {\bfseries 112}, Birkhauser
  Boston, Inc., Boston, MA, 1993.

\bibitem{DLM1997}
C.~Dong, H.~S. Li and G.~Mason, \textsl{Regularity of rational vertex operator
  algebras}, Adv. Math. \textbf{132} (1997), 148--166.

\bibitem{DN1999-1}
C.~Dong and K.~Nagatomo, \textsl{Classification of irreducible modules for the
  vertex operator algebra {$M(1)^{+}$}}, J. Algebra \textbf{216} (1999),
  384--404.

\bibitem{DN2001}
C.~Dong and K.~Nagatomo, \textsl{Classification of irreducible modules for the
  vertex operator algebra {$M(1)^{+}$}:{II}. higher rank}, J. Algebra
  \textbf{240} (2001), 289--325.

\bibitem{FLM}
I.~B. Frenkel, J.~Lepowsky and A.~Meurman, \textsl{Vertex operator algebras and
  the monster}, Pure and Applied Math. {\bfseries 134}, Academic Press, 1988.

\bibitem{Frenkel-Zhu1992}
I.~B. Frenkel and Y.~Zhu, \textsl{Vertex operator algebras associated to
  representations of affine and virasoro algebras}, Duke Math. J. \textbf{66}
  (1992), 123--168.

\bibitem{Kac1998}
V.~Kac, \textsl{Vertex algebras for beginners, second edition}, University
  Lecture Series {\bfseries 10}, American Mathematical Society, 1998.

\bibitem{LL}
J.~Lepowsky and H.~S. Li, \textsl{Introduction to vertex operator algebras and
  their representations}, Progress in Mathematics {\bfseries 227}, Birkhauser
  Boston, Inc., Boston, MA, 2004.

\bibitem{Levitzki1935}
J.~Levitzki, \textsl{On automorphisms of certain rings}, Ann. of Math. (2)
  \textbf{36} (1935), 984--992.

\bibitem{Li1996}
H.~S. Li, \textsl{Local systems of vertex operators, vertex superalgebras and
  modules}, J. Pure Appl. Algebra \textbf{109} (1996), 143--195.

\bibitem{Montgomery1980}
S.~Montgomery, \textsl{Fixed rings of finite automorphism groups of associative
  rings}, Lecture Notes in Mathematics {\bfseries 818}, Springer, Berlin, 1980.

\bibitem{Risa/Asir}
Risa/Asir, \url{http://www.math.kobe-u.ac.jp/Asir/asir.html}.

\bibitem{Tanabe2019}
K.~Tanabe, \textsl{The irreducible weak modules for the fixed point subalgebra
  of the vertex algebra associated to a non-degenerate even lattice by an
  automorphism of order {$2$}}, \url{https://arxiv.org/abs/1910.07126}.

\bibitem{Tanabe2017}
K.~Tanabe, \textsl{Simple weak modules for the fixed point subalgebra of the
  heisenberg vertex operator algebra of rank {$1$} by an automorphism of order
  {$2$} and whittaker vectors}, Proc. Amer. Math. Soc. \textbf{145} (2017),
  4127--4140.

\bibitem{Tanabe2020}
K.~Tanabe, \textsl{Simple weak modules for some subalgebras of the {H}eisenberg
  vertex algebra and {W}hittaker vectors}, Algebr. Represent. Theory
  \textbf{23} (2020), 53--66.

\bibitem{TY2007}
K.~Tanabe and H.~Yamada, \textsl{The fixed point subalgebra of a lattice vertex
  operator algebra by an automorphism of order three}, Pacific J. Math.
  \textbf{230} (2007), 469--510.

\bibitem{TY2013}
K.~Tanabe and H.~Yamada, \textsl{Fixed point subalgebras of lattice vertex
  operator algebras by an automorphism of order three}, J. Math. Soc. Japan
  \textbf{65} (2013), 1169--1242.

\bibitem{Wang1993}
W.~Wang, \textsl{Rationality of virasoro vertex operator algebras},
  International Mathematics Research Notices \textbf{1993} (1993), 197--211.

\bibitem{Yamsk2009}
G.~Yamskulna, \textsl{Rationality of the vertex algebra {$V_{L}^{+}$} when
  {$L$} is a non-degenerate even lattice of arbitrary rank}, J. Algebra
  \textbf{321} (2009), 1005--1015.

\bibitem{Z1996}
Y.~Zhu, \textsl{Modular invariance of characters of vertex operator algebras},
  J. Amer. Math. Soc. \textbf{9} (1996), 237--302.

\end{thebibliography}
